\documentclass[a4paper,reqno,oneside,10pt]{amsart} 
 \pdfoutput=1

\usepackage{pinlabel}
\usepackage{graphicx}

\usepackage{amsmath}
\usepackage{amssymb}
\usepackage{enumerate}
\usepackage{mathrsfs}

\newcommand{\C}{\mathbf{C}}
\newcommand{\R}{\mathbf{R}}
\newcommand{\Q}{\mathbf{Q}}
\newcommand{\Z}{\mathbf{Z}}
\newcommand{\N}{\mathbf{N}}

\newcommand{\T}{\mathbf{T}}
\newcommand{\Hyp}{\mathbf{H}}
\newcommand{\proj}{\mathbf{P}}
\newcommand{\Sph}{\mathbf{S}}

\def\a{\alpha}
\def\b{\beta}
\def\g{\gamma}
\def\G{\Gamma}
\def\d{\delta}
\def\e{\varepsilon}
\def\l{\lambda}
\def\f{\varphi}

\def\s{\sigma}

\newcommand{\Isom}{\mathbf{Isom}}

\newcommand{\fix}{\mathrm{Fix}}

\newcommand{\real}{\mathrm{Re}}
\newcommand{\im}{\mathrm{Im}}

\newcommand{\Mod}{\mathrm{Map}}
\newcommand{\teich}{\mathcal{T}}

\newcommand{\simple}{\mathscr{S}}
\newcommand{\ML}{\mathcal{ML}}

\newcommand{\PML}{\mathcal{PML}}
\newcommand{\curve}{\mathscr{C}}
\newcommand{\mcurve}{\mathcal{S}}
\newcommand{\current}{\mathcal{C}}
\newcommand{\pcurrent}{\mathcal{PC}}

\newcommand{\muth}{\mu_{Th}}

\newcommand{\moduli}{\mathcal{M}}
\newcommand{\qmoduli}{\mathcal{QM}}
\newcommand{\uqmoduli}{\mathcal{Q}^1\mathcal{M}}
\newcommand{\qteich}{\mathcal{QT}}
\newcommand{\uqteich}{\mathcal{Q}^1\mathcal{T}}

\newcommand{\Rep}{\mathsf{Rep}}
\newcommand{\Hgy}{\mathrm{H}}

\newcommand{\Ad}{\mathsf{Ad}}
\newcommand{\SL}{\mathsf{SL}}
\newcommand{\PSL}{\mathsf{PSL}}

\newcommand{\sys}{\textnormal{sys}}
\newcommand{\diff}{\mathrm{d}}
\newcommand{\vol}{\textnormal{vol}}

\newcommand{\cone}{\mathcal{B}}
\newcommand{\hol}{\textnormal{hol}}

\newcommand{\Geod}{\mathcal{G}}
\newcommand{\supp}{\mathrm{supp}}

\newcommand{\mbf}{\mathbf}
\newcommand{\mrm}{\mathrm}
\newcommand{\mc}{\mathcal}

\theoremstyle{plain}
\newtheorem{theorem}{Theorem}[section] 
\newtheorem*{theoremnonumber}{Theorem}

\newtheorem*{proposition123}{Proposition 12.3}
\newtheorem{corollary}[theorem]{Corollary}
\newtheorem{proposition}[theorem]{Proposition}

\newtheorem{lemma}[theorem]{Lemma}

\theoremstyle{definition}

\newtheorem{question}{Question}[section]
\newtheorem{conjecture}{Conjecture}[section]

\newtheorem{remark}{Remark}[section]

\newtheorem*{remarknonumber}{Remark}

\title[]{What's wrong with the growth of simple closed geodesics on nonorientable hyperbolic surfaces}
\date{\today}
\thanks{This work was supported by fund FIRB 2010 (RBFR10GHHH$_{003}$)}

\begin{document}

\maketitle

\begin{flushleft} 
            \textbf{Matthieu Gendulphe}\\
  \begin{small}Dipartimento di Matematica\\
                      Universit\`a di Pisa
  \end{small}
\end{flushleft}
\vspace{1cm}

\begin{abstract}
 A celebrated result of Mirzakhani states that, if $(S,m)$ is a finite area \emph{orientable} hyperbolic surface, then the number of simple closed geodesics of length less than $L$ on $(S,m)$ is asymptotically equivalent to a positive constant times $L^{\dim\ML(S)}$, where $\ML(S)$ denotes the space of measured laminations on $S$. We observed on some explicit examples that this result does not hold for \emph{nonorientable} hyperbolic surfaces. The aim of this article is to explain this surprising phenomenon.\par

 Let $(S,m)$ be a finite area \emph{nonorientable} hyperbolic surface. We show that the set of measured laminations with a closed one--sided leaf has a peculiar structure. As a consequence, the action of the mapping class group on the projective space of measured laminations is not minimal. We determine a partial classification of its orbit closures, and we deduce that the number of simple closed geodesics of length less than $L$ on $(S,m)$ is negligible compared to $L^{\dim\ML(S)}$. We extend this result to general multicurves.\par

 Then we focus on the geometry of the moduli space. We prove that its Teichm\"uller volume is infinite, and that the Teichm\"uller flow is not ergodic. We also consider a volume form introduced by Norbury. We show that it is the right generalization of the Weil--Petersson volume form. The volume of the moduli space with respect to this volume form is again infinite (as shown by Norbury), but the subset of hyperbolic surfaces whose one--sided geodesics have length at least $\e>0$ has finite volume.\par
 
These results suggest that the moduli space of a nonorientable surface looks like an infinite volume geometrically finite orbifold. We discuss this analogy and formulate some conjectures.
\end{abstract}

\setcounter{tocdepth}{1}
\newpage

\section{Introduction}\label{sec:introduction}

 In this article we are interested in three aspects of the geometry of hyperbolic surfaces and their moduli spaces:
\begin{itemize} 
\item  the growth of the number of closed geodesics of a given topological type,
\item the mapping class group action on the space of measured laminations,
\item the Teichm\"uller and Weil--Petersson volumes of the moduli space.
\end{itemize}
The recent work of Mirzakhani brought to light many connections between these topics, in addition to solving important problems. However her work only deals with orientable surfaces.\par
 The purpose of this article is to point out some interesting phenomena that occur in the case of nonorientable surfaces. They are interesting for two reasons. Firstly they show a difference between the orientable and the nonorientable cases. Secondly they suggest that the moduli space of a nonorientable surface looks like an infinite volume geometrically finite orbifold. Before going to this conclusion let us describe our results and emphasize the differences with the orientable setting.\par
 
  In this introduction $(S,m)$ is a finite area hyperbolic surface without boundary.
Given a homotopy class of closed curves $\g$, we define its $m$--length $\ell_m(\g)$ as the infimum of the lengths of its representatives. If $\g$ is nontrivial and nonperipheral, then $\ell_m(\g)$ is realized by its unique geodesic representative. The mapping class group $\Mod(S)$ acts on the set of homotopy classes, we denote by $\mcurve_{\g_0} =  \Mod(S)\cdot\g_0$ the orbit of a homotopy class $\g_0$. 
   
\subsection*{Growth of simple closed geodesics}
Let us start with the result that motivated our study. In \cite{mirzakhani-preprint} Mirzakhani established:
 
 \begin{theoremnonumber}[Mirzakhani]
 Let $(S,m)$ be a finite area orientable hyperbolic surface. For any simple closed geodesic $\g_0$ of $(S,m)$ there exists $c(m,\g_0)>0$ such that 
 \begin{eqnarray*}
\lim_{L\rightarrow +\infty}   \frac{\left| \left\{\g\in \mcurve_{\g_0}~;~\ell_m(\g)\leq L \right\} \right| }{L^{\dim\ML(S)}} &=& c(m,\g_0).
\end{eqnarray*}
\end{theoremnonumber} 

\begin{remarknonumber}
The case of the punctured torus is due to McShane and Rivin (\cite{mcshane}).
\end{remarknonumber}

 We denote by $\ML(S)$ the space of measured laminations of $S$, its dimension is given by:
 $$\dim\ML(S)~=~\left\{\begin{array}{ll} 6g-6+2r & \textnormal{if $S$ is orientable} \\
                                                      3g-6+2r & \textnormal{if $S$ is nonorientable} 
                              \end{array}\right. ,$$
where $g$ is the genus of the surface and $r$ its number of punctures.\par
In contrast to Mirzakhani's result we show:

\begin{theorem}\label{thm:1}
 Let $(S,m)$ be a finite area nonorientable hyperbolic surface. For any simple closed geodesic $\g_0$ of $(S,m)$ we have 
 \begin{eqnarray*}
\lim_{L\rightarrow +\infty}   \frac{\left| \left\{\g\in \mcurve_{\g_0}~;~\ell_m(\g)\leq L \right\} \right| }{L^{\dim\ML(S)}} &=& 0.
\end{eqnarray*}
 \end{theorem}
 
 This phenomenon was already known in a few cases.  In \cite{gendulphe} we treated the case of nonorientable surfaces with Euler characteristic  $-1$. Based on Mirzakhani's result, we showed that $\left| \left\{\g\in \mcurve_{\g_0}~;~\ell_m(\g)\leq L \right\} \right|$ is asymptotically equivalent to a monomial whose degree is an integer less than $\dim \ML(S)$. In \cite{huang} Huang and Norbury studied the case of the thrice--punctured projective plane denoted by $N_{1,3}$. They related the growth of the number of one--sided simple closed geodesics to the growth of the number of Markoff quadruples. In the recent work \cite{gamburd}, Gamburd, Magee and Ronan determined the asymptotic of the growth of the number of integral Markoff tuples. Let $m_0$ be the most symmetric hyperbolic metric (up to isometry) on the thrice--punctured projective plane, and let $\mcurve_{\g_0}$ be the orbit of the one--sided simple closed geodesics. A combination of both works entails the existence of $\b\in (2,3)$ such that the limit
$$
\lim_{L\rightarrow +\infty}   \frac{\left| \left\{\g\in \mcurve_{\g_0}~;~\ell_{m_0}(\g)\leq L \right\} \right| }{L^{\b}} 
$$
exists and is positive. This result has been extend to any hyperbolic metric on $N_{1,3}$ by Magee (\cite{magee}).\par

 \subsection*{Counting measures} 
 To prove her theorem, Mirzakhani introduced a family of counting measures $(\nu_{\g_0}^L)_{L>0}$ on $\ML(S)$, and showed its weak$^\ast$ convergence towards a positive multiple of the Thurston measure $\muth$. This convergence can be interpreted as the equidistribution of $\mcurve_{\g_0}$ in the projective space of measured laminations $\PML(S)$.\par

  The framework of Mirzakhani's proof is avaible for nonorientable surfaces as well as for orientable ones. Actually the difference between the two theorems reflects a difference in the dynamics of the action of $\Mod(S)$ on $\ML(S)$. In the nonorientable case, the orbit $\mcurve_{\g_0}$ accumulates on the subset $\ML^+(S)\subset\ML(S)$  of measured laminations without one--sided closed leaves, which is negligible with respect to the Thurston measure $\muth$ (Danthony and Nogueira \cite{nogueira}). This implies that the family of counting measures $(\nu^L_{\g_0})_{L>0}$ weak$^\ast$ converges towards the zero measure (Proposition~\ref{pro:convergence-measure}).
  
 \subsection*{Dynamics of the mapping class group action} 
Let us focus on the action of $\Mod(S)$ on $\ML(S)$. When $S$ is orientable this action is rather well--understood:
\begin{itemize}
\item the action of $\Mod(S)$ on $\PML(S)$ is minimal (Thurston, see \cite{flp}),
\item the action of $\Mod(S)$ on $\ML(S)$ is ergodic with respect to $\muth$ (Masur \cite{masur}),
\item there is a classification of $\Mod(S)$--invariant locally finite ergodic measures on $\ML(S)$ (Hamenst\"adt \cite{hamenstadt}, Lindenstrauss--Mirzakhani \cite{mirzakhani-imrn}),
\item there is a classification of the orbit closures of $\Mod(S)$ in $\ML(S)$ (\emph{ibid.}).
\end{itemize}
 In contrast to these results we show that:
\begin{itemize}
\item the action of $\Mod(S)$ on $\PML(S)$ is not minimal (Proposition~\ref{pro:consequences}),
\item the action of $\Mod(S)$ on $\PML(S)$ is topologically transitive if and only if $S$ is of genus one (Proposition~\ref{pro:transitivity}),
\item we have a partial classification of the orbit closures of $\Mod(S)$ on $\PML(S)$, which is complete when $S$ has genus one (Theorem~\ref{thm:closures}).  
\item the action of $\Mod(S)$ on $\ML(S)$ is not topologically transitive, thus is not ergodic with respect to $\muth$ (Proposition~\ref{pro:consequences}),
\end{itemize} 
We also determine the unique minimal invariant closed subset of $\PML(S)$ (\textsection\ref{sec:orbit-closure}).

 \subsection*{Structure of $\ML(S)$}
These results are consequences of the peculiar structure of $\ML(S)$ when $S$ is nonorientable. As mentioned above, the set $\ML^-(S)\subset\ML(S)$ of measured laminations with a one--sided closed leaf has full Thurston measure (Danthony and Nogueira \cite{nogueira}). In addition it admits a canonical cover by topological open balls of dimension $\dim\ML(S)$ (\textsection\ref{sec:decomposition}). Each ball $\cone(\g)$ is associated to a collection of disjoint one--sided simple closed geodesics $\g=\{\g_1,\ldots,\g_k\}$, and consists in the set of measured laminations that admit $\g_1,\ldots,\g_k$ as leaves. The projection of $\cone(\g)$ into $\PML(S)$ is a topological open ball. These balls form a packing (\emph{i.e.} are pairwise disjoint) when $S$ has genus one.

\subsection*{Volumes of moduli spaces}
The particular structure of $\PML(S)$ influences the geometry of the moduli space $\moduli(S)$. This is logical since $\PML(S)$ is the Thurston boundary of the Teichm\"uller space.\par

 When $S$ is nonorientable, we show that the Teichm\"uller flow is not ergodic and that the Teichm\"uller volume of the moduli space is infinite (\textsection\ref{sec:teichmuller}). We also consider a kind of volume form introduced by Norbury (\cite{norbury}) as a generalization of the Weil--Petersson volume form to the case of nonorientable surfaces. We show (\textsection\ref{sec:weil-petersson}) that it is indeed the right generalization from the point of view of the twist flow. Then we provide a simpler proof of the fact --- due to Norbury (\emph{ibid}.) --- that $\moduli(S)$ has infinite volume with respect to this volume form. We also show that the subset $\{\sys^-\geq \e\}$ of points in $\moduli(S)$ whose one--sided simple closed geodesics have length at least $\e>0$ is a finite volume deformation retract of $\moduli(S)$ (\textsection\ref{sec:finite-convex}).
 
\subsection*{Counting multicurves }
Mirzakhani (\cite{mirzakhani-preprint}) has extended her theorem on the growth of simple closed geodesics to general multicurves. The particular case of the punctured torus has been solved independently by Erlandsson and Souto (\cite{souto}), who also established some other interesting results. We first need few lines to precise the terminology.\par

 A \emph{multicurve} is a formal $\R_+^\ast$--linear combination $\g=a_1\g_1+\ldots +a_n\g_n$ of distinct homotopy classes of noncontractible and nonperipheral closed curves. For technical reasons we assume that each $\g_i$ is \emph{primitive}, that is corresponds to the conjugacy class of primitive elements of $\pi_1(S)$. We say that $\g$ is an \emph{integral} (resp. \emph{rational}) multicurve if moreover $a_i\in\N$ (resp. $a_i\in\Q$) for any $i=1,\ldots,n$. We denote by $\mcurve$ the set of integral multicurves of $S$. The mapping class group acts on $\mcurve$, we denote by $\mcurve_{\g}=\Mod(S)\cdot \g$ the orbit of $\g$. Given a hyperbolic metric $m$ on $S$, we define the $m$--length of $\g$ by $\ell_m(\g)=a_1\ell_m(\g_1)+\ldots+a_n\ell_m(\g_n)$. We say that a multicurve is \emph{simple} if its components are simple and disjoint.\par
 
  For sake of clarity, we stated Mirzakhani's theorem in the case of a simple closed geodesic, but it applies to any \emph{simple integral multicurve}. As we mentioned above, she has extended her theorem to all multicurves (\cite[Theorem~1.1]{mirzakhani-preprint}). We do the same way with our:
 
\begin{theorem}\label{thm:2}
Let $(S,m)$ be a finite area nonorientable hyperbolic surface. For any integral multicurve $\g_0$ we have 
\begin{eqnarray*}
\lim_{L\rightarrow 0}   \frac{\left| \left\{\g\in \mcurve_{\g_0}~;~\ell_m(\g)\leq L \right\} \right| }{L^{\dim\ML(S)}} &=& 0.
\end{eqnarray*}
\end{theorem}
 
 We prove Theorem~\ref{thm:2} in the same manner as Theorem~\ref{thm:1} by showing that the family of counting measures $(\nu_{\g_0}^{L})_{L>0}$ converges to the zero measure. These measures are not defined on $\ML(S)$ anymore, but on the space  $\current(S)$ of geodesic currents. We show that any limit point of $(\nu_{\g_0}^{L})_{L>0}$ is supported on $\ML^+(S)$, and we conclude using a powerful result of Erlandsson and Souto (\cite{souto}) which states that limit point of $(\nu_{\g_0}^{L})_{L>0}$ is absolutely continuous with respect to $\muth$.\par
 
 The method developed in \cite{mirzakhani-preprint} is very different, and does not make use of geodesic currents. Working with geodesic currents present the advantage that Theorem~\ref{thm:2} extends immediately to all geometric structures that define a filling geodesic currents, like complete negatively curved metrics (see Theorem~\ref{thm:final} for a precise statement).
 
 Let us mention a result which explains why an accumulation point in $\current(S)$ of an orbit $\simple_{\g_0}$ can not have a one--sided leaf: 
 
 \begin{proposition123}
 Let $\l\in\ML^-(S)$ be a measured lamination with a one--sided closed leaf $\g$. For any $k\geq 1$, there exists a neighborhood $U_k$ of $\l$ in $\current(S)$ such that for any geodesic current $c\in U_k$ there exists a geodesic $\d\in\supp(c)$ that projects either on $\g$ either on a geodesic with $k$ self--intersections.
 \end{proposition123}
 
\section{Conclusion}\label{sec:conclusion}

 The analogy between moduli spaces and locally symmetric orbifolds has guided the study of Teichm\"uller spaces and mapping class groups since many years. So far, moduli spaces were compared to \emph{finite volume} locally symmetric orbifolds. The results established in this paper suggest that, in the case of a nonorientable surface, the moduli space looks like a \emph{geometrically finite} locally symmetric orbifold of \emph{infinite volume}. The aim of this section is to defend this point of view.\par
 
  For sake of simplicity, we compare the mapping class group $\Mod(S)$ with a geometrically finite Kleinian group $\G$. We mostly focus on the dynamical properties of their actions. In the first paragraphs we recall some basic results and definitions. Then we describe the analogy and state some conjectures. Such an analogy has been particularly fruitful in the orientable setting (see \cite{ABEM,EM}).

\subsection*{Geometrically finite Kleinian groups}
We recall that a \emph{Kleinian group} $\G$ is a discrete subgroup of $\Isom(\Hyp^n)$. Its \emph{limit set} $\Lambda$ is the set of points in $\partial\Hyp^n$ that are adherent to any orbit of $\G$ in $\Hyp^n$. We assume that $\G$ is \emph{non elementary}, which means that $\Lambda$ is infinite. The \emph{convex core} of the orbifold $\Hyp^n/\G$ is the projection $\mrm{convex}(\Lambda)/\G$ of the convex hull of $\Lambda$ in $\Hyp^n$.\par
 
Following \cite{kapovich}, we say that $\G$ is \emph{geometrically finite} if there exists $\e>0$ such that the $\e$--tubular neighborhood of the convex core has finite volume. For example,
a Kleinian group which admits a convex fundamental domain bounded by finitely many geodesic faces is geometrically finite, but the converse is not true. Note that a geometrically finite Kleinian group is a lattice (\emph{i.e.} has finite covolume) if and only if $\Lambda=\partial\Hyp^n$.\par

 There are various characterizations of geometrical finiteness, let us mention the following: a Kleinian group $\G$ is geometrically finite if and only if there exists $\{B_i\}_{i\in I}$ a $\G$--invariant collection of disjoint horoballs centered at fixed points of parabolic subgroups of $\G$ such that 
 $\left(\mrm{convex}(\Lambda)-\cup_i B_i\right)/\G$ 
is compact  (see \cite{kapovich} or \cite[Chap.12]{ratcliffe}).

\subsection*{Patterson--Sullivan theory}
Let $\G$ be a Kleinian group. To study the dynamics of the action of $\G$ on $\partial\Hyp^n$ one naturally looks for an invariant Radon measure supported on $\Lambda$. Such a measure is necessarily trivial when $\G$ is non elementary, therefore one considers a more general object called conformal density.\par

  A \emph{conformal density of dimension $\d>0$} is a family $\{\mu_x\}_{x\in\Hyp^n}$ of Radon measures on $\partial \Hyp^n$ such that any two measures $\mu_x$ and $\mu_{y}$ are absolutely continuous one with respect to the other, and are related by the following equality:
\begin{eqnarray*}
\frac{\mrm d \mu_y}{\mrm d \mu_x}(\xi) & = & e^{-\d \b_\xi(y,x)}\quad(\forall\xi\in\partial\Hyp^n),
\end{eqnarray*}
where $\b_\xi(y,x)$ is the Busemann cocycle that gives the signed distance between the horospheres centered at $\xi$ passing through $y$ and $x$. We say that the conformal density is \emph{$\G$--invariant} if $\g_\ast \mu_x = \mu_{\g x}$ for any $\g\in\G$ and any $x\in\Hyp^n$. The Lebesgue measure on $\Sph^{n-1}$ induces an invariant conformal density of dimension $n-1$. More interestingly, Patterson and Sullivan (\cite{sullivan}) constructed a non zero invariant conformal density supported on $\Lambda$, and they showed that its dimension is equal to the \emph{critical exponent} of $\G$: 
\begin{eqnarray*}
\d_\G &  = &   \inf\left\{\d>0~;~\sum_\G e^{-\d d_{\Hyp^n}(o,\g o)}<+\infty\right\},
\end{eqnarray*}
 where $o$ is any point in $\Hyp^n$. \par

 To any $\G$--invariant conformal density corresponds a measure on the unit tangent bundle of $\Hyp^n$ which is invariant under the geodesic flow and the action of $\G$, this is the associated \emph{Bowen--Margulis--Sullivan measure} (\cite{sullivan-acta}). In the case of the conformal density induced by the Lebesgue measure on $\Sph^{n-1}$, the associated Bowen--Margulis--Sullivan measure is simply the Liouville measure. The properties of a conformal density and its Bowen--Margulis--Sullivan measure are related by the Hopf--Tsuji--Sullivan dichotomy (see \cite[Th\'eor\`eme~1.7]{roblin}).\par
 
 Let us assume that $\G$ admits a convex fundamental domain bounded by finitely many geodesic faces. Then the $\G$--invariant conformal density of dimension $\d_\G$ is unique up to scaling, and has finite total mass. Moreover $\d_\G$ is equal to the Hausdorff dimension of $\Lambda\subset\partial \Hyp^n$ with respect to the angular distance at any point of $\Hyp^n$ (\cite{sullivan-acta}). The associated Bowen--Margulis--Sullivan measure on $\Hyp^n/\G$ is ergodic. This measure is a crucial tool to establish growth and equidistribution results (see \cite{roblin}).

\subsection*{The analogy}
 Let $S$ be a surface of finite type with $\chi(S)<0$. The Teichm\"uller space equipped with the Teichm\"uller metric plays the role of $\Hyp^n$. The mapping class group acts properly and discontinuously by isometry on $\teich(S)$, like a Kleinian group on $\Hyp^n$. The Thurston boundary replaces $\partial\Hyp^n$, even though $\PML(S)$ is not the visual boundary of the Teichm\"uller metric (Kerckhoff \cite{kerckhoff}).\par

 The Thurston measure induces a conformal density $\{\mu_{X}\}_{X\in\teich(S)}$ analogous to the conformal density induced by the Lebesgue measure, it is given by
\begin{eqnarray*}
 \mu_X(U) & = & \muth(B_X\cap \pi^{-1}(U)),
 \end{eqnarray*}
for any $U\subset\PML(S)$ measurable, where $\pi:\ML(S)\rightarrow \PML(S)$ is the canonical projection and $B_X=\{\l\in\ML(S)~;~\ell_X(\l)\leq 1\}$. The Thurston measure is also related to the Teichm\"uller volume defined on  the space of quadratic differentials (see \textsection\ref{sec:teichmuller}). This space identifies canonically with the cotangent bundle of $\teich(S)$, so that the Teichm\"uller volume appears as an analogue of the Liouville measure. We refer to \cite[\textsection 2.3.1]{ABEM} for more details.\par 

 Now we assume that $S$ is nonorientable. The mapping class group shouldn't be considered as a lattice since it has infinite covolume. We suggest to compare it with a geometrically finite Kleinian group. To motivate this analogy we first look at its limit set. In \textsection\ref{sec:orbit-closure} we prove the following conjecture in the case of genus one surfaces:

\begin{conjecture}
The limit set of $\Mod(S)$ in $\PML(S)$ is the set of projective measured laminations without one--sided closed leaves, it is denoted by $\PML^+(S)$.
\end{conjecture}

 We have already mentioned that $\PML^+(S)$ is negligible with respect to the Thurston measure class (Danthony and Nogueira \cite{nogueira}). Similarly, in dimension three, the limit set of an infinite covolume geometrically finite Kleinan group is negligible with respect to the Lebesgue measure class (Ahlfors \cite{ahlfors}).\par

 We observe that $\PML^+(S)$ coincides with the set of points in the Thurston boundary that are adherent to $\{\sys^-\geq\e\}\subset\teich(S)$ for $\e>0$ small enough, where $\sys^-$ denotes the length of the shortest closed one--sided geodesic. We believe that $\{\sys^-\geq\e\}\subset\teich(S)$ is the analogue of an infinite convex polyhedron of $\Hyp^n$.\par
  
  Let us write $\{\sys^-\geq\e\} = \cap_{\g\in\simple^-} \{\ell_\g\geq \e\}$ where we denote by $\simple^-$ the set of isotopy classes of one--sided simple closed curves. When $\g$ is a two--sided simple closed geodesic, the hypersurface $\{\ell_\g=\e\}$ is a kind of horosphere. When $\g$ is a one--sided simple closed geodesic, the boundary of $\{\ell_\g=\e\}$ in the Thurston compactification is a polyhedral sphere (see \textsection\ref{sec:decomposition}), like the boundary of a geodesic hyperplane in $\Hyp^n$ is a conformal sphere in $\partial\Hyp^n$. So it seems natural to think about $\{\ell_\g=\e\}$ as a geodesic hyperplane (or maybe as a hypersurface made of points at a given distance from a geodesic hyperplane).\par
  
 Then, the fact that $\{\sys^-\geq\e\}\subset\moduli(S)$ has finite Weil--Petersson volume (\textsection\ref{sec:finite-convex}) is analogous to geometrical finiteness. Note that the tubular neighborhood of $\{\sys^-\geq\e\}$ with respect to the Teichm\"uller metric is contained in $\{\sys^-\geq\e_1\}$ for some $\e_1>0$ small enough (use the main theorem of \cite{minsky}).\par
 
  Let us examine the characterization of geometrical finiteness we mentioned in a previous paragraph.
We have $\{\sys^- \geq \e\} - \bigcup_{\g\in\simple^+} \{\ell_\g<\e\} =  \{\sys\geq \e\}$ where $\simple^+$ is the set of isotopy classes of two--sided simple closed curves. As well--known, the subset $\{\sys\geq \e\}\subset\moduli(S)$ is compact, and any $\{\ell_\g<\e\}\subset \teich(S)$ with $\g\in\simple^+$ is a kind of horoball. These horoballs are not pairwise disjoint, but any two of them are disjoint if their corresponding geodesics intersect (Collar Lemma). So the characterization of geometrical finiteness is somehow satisfied.\par

 We believe that $\Mod(S)$ should be compared to a Kleinian group that admits a convex fundamental domain bounded by finitely many faces. Such fundamental domains have been constructed when $S$ has a small Euler characteristic (the author \cite{gendulphe}, Huang and Norbury \cite{huang}).\par
 
  We now formulate some optimistic conjectures. The following conjecture is the analogue of Sullivan's theorem (\cite[Theorem~1]{sullivan-acta}): 

\begin{conjecture}\label{conj:measure}
There exists a unique (up to scaling) ergodic $\Mod(S)$--invariant Radon measure whose support is $\ML^+(S)$, it is $\d$--homogeneous where $\d$ is the Hausdorff dimension of $\ML^+(S)$.
\end{conjecture}

The Hausdorff dimension of $\ML^+(S)$ is the one defined by any Euclidean norm in any train--track chart. The transition maps between the train--track charts are piecewise linear  --- in particular Lipschitz --- therefore the Hausdorff dimension does not depend on the choice of a chart.\par

 Let us now consider the growth of simple closed geodesics. The conjecture below is true when the Euler characteristic of $S$ is $-1$ (see \cite{gendulphe}), and seems to be true when $S$ is the thrice--punctured projective plane (Magee \cite{magee}):  
 
\begin{conjecture}\label{conj:growth}
For any hyperbolic metric $m$ on $S$, and for any simple closed geodesic $\g_0$, there exists a constant $c>0$ such that $|\{\g\in\mcurve_{\g_0}~;~\ell_m(\g)\leq L\}|\simeq c L^\d$ as $L$ tends to infinity, where $\d$ is the Hausdorff dimension of $\ML^+(S)$.
\end{conjecture}

\subsection*{A remarkable example}
Let $N_{1,3}$ denote the thrice--punctured projective plane. In a forthcoming paper (\cite{gendulphe-moduli}) we show that the action of $\Mod(N_{1,3})$ on the Thurston compactification of $\teich(N_{1,3})$ is topologically conjugate to the action of a geometrically finite Kleinian group $\G$ on $\Hyp^3\cup\partial\Hyp^3$. A finite index subgroup of $\G$ is the so--called Apollonian group, that is the group of M\"obius transformations that preserve the Apollonian packing. Many counting problems related to the Apollonian packing have been studied recently (see \cite{Oh}).

\section{Organization of the paper}

The rest of paper falls into three parts. The first part is devoted to the space of measured laminations. We start with some explicit examples (\textsection\ref{sec:complexity}) before describing the peculiar structure of $\ML^-(S)$ in the general case (\textsection\ref{sec:decomposition}).  Then we study the dynamics of the mapping class group action (\textsection\ref{sec:orbit-closure}) and prove Theorem~\ref{thm:1} (\textsection\ref{sec:growth-simple}). In the second part we extend Theorem~\ref{thm:1} to multicurves. The main difficulty is to show that the orbit of a multicurve in $\current(S)$ accumulates on $\ML^+(S)$ (\textsection\ref{sec:accumulation-currents}). In the third part we consider various aspects of the geometry of Teichm\"uller spaces. We show that the Teichm\"uller volume of the moduli space of unit area quadratic differentials is infinite, and that the Teichm\"uller flow is not ergodic (\textsection\ref{sec:teichmuller}). Then we discuss the definition of the Weil--Petersson volume (\textsection\ref{sec:weil-petersson}), and we show that the subset $\{\sys^-\geq \e\}$ of the moduli space has finite volume (\textsection\ref{sec:finite-convex}).

\tableofcontents

\section{Acknowledgements}
I would like to thank Juan--Carlos \'Alvarez Paiva and Gautier Berck for an original idea of title (see \cite{berck}).
I would like to thank Gabriele Mondello who gave me interesting questions and ideas, a part of this work was realized when I was in Rome. I would like to thank the members of the Mathematics department of the university of Pisa for their friendly welcome. Finally I thank the Italian Republic for the last four years of financial support.

\section{Notations}
For the convenience of the reader we list some notations introduced in the paper:
\begin{itemize}
\item[$B_X(L)$] set of measured laminations $\l\in\ML(S)$ with $\ell_X(\l)\leq L$ (\textsection\ref{sec:decomposition})
\item[$b_X(L)$] Thurston measure of $B_X(L)$ (\textsection\ref{sec:decomposition})
\item[$\cone(\g)\subset\ML(S)$] set of measured laminations that admit the components of the simple multicurve $\g=\g_1+\ldots+\g_k$ as closed leaves (\textsection\ref{sec:decomposition})
\item[$\curve(S)$] curve complex of $S$
\item[$\curve^-(S)$] subcomplex of one--sided curves (\textsection\ref{sec:decomposition})
\item[$S$] smooth connected surface of finite type with negative Euler characteristic
\item[$N_{g,n}$] compact nonorientable surface of genus $g$ with $n$ boundary components (\textsection\ref{sec:complexity})
\item[$\simple$] set of isotopy classes of simple closed curves that do not bound a disk, a punctured disk, an annulus or a M\"obius strip (\textsection\ref{sec:preliminaries})
\item[$\simple^-\subset\simple$] set of isotopy classes of one--sided simple closed curves (\textsection\ref{sec:preliminaries})
\item[$\simple^-_{b}\subset\simple^-$] set of isotopy classes of one--sided simple closed curves whose complement is orientable (\textsection\ref{sec:preliminaries})
\item[$\simple^-_{nb}\subset\simple^-$] set of isotopy classes of one--sided simple closed curves whose complement is nonorientable (\textsection\ref{sec:preliminaries})
\item[$\mcurve$] set of integral multicurves (\textsection\ref{sec:introduction})
\item[$\mcurve_\g$] $\Mod(S)$--orbit of an integral multicurve $\g\in\mcurve$ (\textsection\ref{sec:introduction})
\item[$\mcurve_k$] set of integral multicurves with exactly $k$ self--intersections (Part~\ref{part:currents})
\item[$\tilde S$] universal cover of $S$
\item[$\tilde S_\infty$] visual boundary of $\tilde S$
\item[$\Geod(\tilde S)$] space of geodesics of $\tilde S$
\item[$\current(S)$] space of geodesic currents on $S$
\item[$\ML(S)$] space of measured laminations on $S$ (\textsection\ref{sec:preliminaries})
\item[$\ML^-(S)$] subspace of measured laminations with a one--sided closed leaf (\textsection\ref{sec:decomposition})
\item[$\ML^+(S)$] subspace of measured laminations without one--sided closed leaves (\textsection\ref{sec:decomposition})
\item[$\ML(S,\Z)$] set of integral simple multicurves (\textsection\ref{sec:preliminaries})
\item[$\ML(S,\Q)$] set of rational multiples of elements of $\ML(S,\Z)$ (\textsection\ref{sec:preliminaries})
\item[$\ML_N^-(S,\Z)$] set of simple multicurves in $\ML(S,\Z)$ whose one--sided components have weight bounded by $N$ (\textsection\ref{sec:growth-simple})
\item[$\PML(S)$] projective space of measured laminations (\textsection\ref{sec:preliminaries})
\item[$\teich(S)$] Teichm\"uller space of $S$ (\textsection\ref{sec:preliminaries})
\item[$\qteich(S)$] space of quadratic differentials on $S$ (\textsection\ref{sec:teichmuller})
\item[$\moduli(S)$] moduli space of $S$ (\textsection\ref{sec:preliminaries})
\item[$\Mod(S)$] mapping class group of $S$ (\textsection\ref{sec:preliminaries})
\item[$\Mod^\ast(S)$] extended mapping class group of $S$ (\textsection\ref{sec:preliminaries})
\item[$\muth$ or $\muth^S$] Thurston measure on $\ML(S)$ (\textsection\ref{sec:preliminaries})
\item[$\mu^L$] counting measure of $\ML(S,\Z)$ (see \eqref{eq:Thurston-measure} in \textsection\ref{sec:preliminaries})
\item[$\mu^L_{\g_0}$] counting measure of the orbit $\mcurve_{\g_0}$
\item[$\sys^-$] the length of the shortest closed one--sided geodesic.
\end{itemize}

\section{Preliminaries}\label{sec:preliminaries}

In this section we recall some classical definitions and results. Most of them deal with measured laminations. We refer to \cite{bonahon-book,flp,hatcher} for a more detailed exposition.\par

In all this text $S$ is a smooth connected surface of finite type with negative Euler characteristic. A \emph{hyperbolic metric} on $S$ is a finite area complete metric of constant curvature $-1$. We assume that the boundary $\partial S$ is geodesic if it is nonempty. Given a hyperbolic metric $m$ on $S$, each homotopy class $\g$ of noncontractible and nonperipheral closed curves admits a unique geodesic representative. We abusively use the same notation for a geodesic and its homotopy class. In the sequel we implicitly assume that a \emph{closed curve} is not homotopic to a point nor a puncture.

\subsection*{Teichm\"uller spaces and mapping class groups}
The \emph{Teichm\"uller space} $\teich(S)$ is the space of isotopy classes of hyperbolic metrics on $S$. We assume that the lengths of the boundary components are fixed. The Teichm\"uller space is a smooth manifold diffeomorphic to an open ball of dimension $-3\chi(S)-n$ where $n$ is the sum of the number of boundary components plus the number of punctures of $S$.\par

The \emph{extended mapping class group} $\Mod^\ast(S)$ is the group of diffeomorphisms of $S$ up to isotopy. The \emph{mapping class group} $\Mod(S)$ is the subgroup of $\Mod^\ast(S)$ that preserves each boundary component, each puncture, and the orientation. The mapping class group acts properly and discontinuously by diffeomorphisms on the Teichm\"uller space. The orbifold $\moduli(S)=\teich(S)/\Mod(S)$ is the \emph{moduli space of} $S$.

\subsection*{Measured laminations}
Let $m$ be a hyperbolic metric on $S$. A $m$--\emph{geodesic lamination} on $S$ is a compact subset of $S$ foliated by simple $m$--geodesics. The boundary components can not be leaves of the lamination.
A \emph{transverse measure} $\mu$ for $\l$ is the data of a Radon measure $\mu_k$ on each arc $k$ transverse to $\l$ such that:
\begin{itemize}
\item the restriction of $\mu_k$ to a transverse subarc $k'\subset k$ is $\mu_{k'}$,
\item an isotopy preserving $\l$ between two transverse arcs $k,k'$ sends $\mu_k$ on $\mu_{k'}$,
\item the support of $\mu_k$ is $k\cap \l$.
\end{itemize}
A \emph{measured lamination} is a $m$--geodesic lamination equipped with a transverse measure. The transverse measure determines the $m$--geodesic lamination. We denote by $\ML(S)$ the space of measured laminations on $S$.\par

Let $\simple$ be the set of isotopy classes of simple closed curves that do not bound a disk, a punctured disk, an annulus or a M\"obius strip. The map $(\l,\mu)\mapsto (\mu(\g))_{\g\in\simple}$ is a proper topological embedding of $\ML(S)$ into the affine space $\R^\simple$ endowed with the product topology. The space $\ML(S)\cup\{0\}$ is homeomorphic to an open ball of the same dimension as $\teich(S)$. The extended mapping class group acts by homeomorphisms on $\ML(S)\cup\{0\}$.
 
\subsection*{The Thurston compactification}
  The space $\ML(S)$ is a cone since one can multiply a transverse measure by a positive number. We denote by $\PML(S)$ the projective space of measured laminations. It is homeomorphic to a sphere and embedds topologically  into $\proj(\R^\simple)$ through the map $[\l,\mu]\mapsto [\mu(\g)]_{\g\in\simple}$.\par
  
 The map $[m]\mapsto[\ell_\g(m)]_{\g\in\simple}$ is a topological embedding of $\teich(S)$ into $\proj(\R^\simple)$.
The closure of its image is a closed ball whose boundary is exactly the image of $\PML(S)$ in $\proj(\R^\simple)$.
The union $\teich(S)\cup\PML(S)$ endowed with the topology induced by $\proj(\R^\simple)$ is called the \emph{Thurston compactification} of $\teich(S)$. The mapping class group acts by homeomorphism on the Thurston compactification.
 
\subsection*{Train--tracks and piecewise integral linear structure}
The space $\ML(S)$ has no canonical smooth structure, however it admits a \emph{piecewise integral linear} structure. This means that it has an atlas whose transition maps are piecewise linear and coincide with integral linear isomorphisms over pieces defined by integral linear inequalities. In particular, there is a well--defined notion of integral point and a canonical measure on $\ML(S)$.\par

 This atlas is canonical and explictly defined through train--tracks. In this article we do not use directly train--tracks, we refer to
 \cite{bonahon-book,hatcher} for more details. These references consider both orientable and nonorientable surfaces. 

\subsection*{Integral points and Thurston measure}
 We denote by $\ML(S,\Z)$ the set of integral points of $\ML(S)$, they correspond to \emph{integral simple multicurves}, that is to linear combinations of the form $a_1\a_1+\ldots+a_k\a_k$ where $a_1,\ldots,a_k\in\N^\ast$ and $\a_1,\ldots, \a_k$ are disjoint elements of $\simple$. In the same way we denote by $\ML(S,\Q)$ the set of rational multiples of elements in $\ML(S,\Z)$.\par
 
  The \emph{Thurston measure} $\muth$ on $\ML(S)$ (also denoted $\muth^S$) corresponds to the Lebesgue measure in any train--track chart. Thus it is also the weak$^\ast$ limit of the counting measures of integral points. To be more precise, let $\mu^L$ ($L>0$) be the measure defined by
\begin{eqnarray}\label{eq:Thurston-measure}
\mu^L &= & \frac{1}{L^{\dim \ML(S)}} \sum_{\g\in \ML(S,\Z)} \mbf 1_{\frac{1}{L}\g}.
\end{eqnarray}
The measure $\mu^L$ tends to $\muth$ in the weak$^\ast$ topology as $L$ tends to infinity. Integral points and the Thurston measure are invariant under the action of $\Mod(S)$.  \par

\subsection*{The intersection form}
The intersection number $i(\g,\d)$ between two simple closed geodesics $\g$ and $\d$ is well--defined. It extends to $\ML(S,\Q)$ by linearity, and then to $\ML(S)$ by density and uniform continuity of $i(\cdot,\cdot)$.

\subsection*{How to look at nonorientable surfaces}
In order to work with nonorientable surfaces it is convenient to look at them as orientable surfaces with boundary submitted to some identifications. There are two kinds of identifications, but we are going to use only one in this article. Let $c$ be a boundary component of a hyperbolic surface $(S,m)$, one can identify the points of $c$ as follows: two points are identified if they divide $c$ into two segments of equal length. The quotient of $(S,m)$ under this identification is a nonorientable hyperbolic surface, and the boundary component $c$ projects onto a one--sided geodesic. Note that this identification does not change the Euler characteristic.\par

 Let us consider two examples. Let $(\T,m)$ be a one--holed torus equipped with a hyperbolic metric. The surface obtained by identification of $\partial\T$ is homeomorphic to the connected sum of three real projective planes. Let $(P,m)$ be a hyperbolic pair of pants, the identification of each boundary component of $P$ produces again a hyperbolic surface homeomorphic to the connected sum of three projective planes.\par

\subsection*{One--sided simple closed curves}
 We denote by $\simple^-$ the subset of one--sided elements of $\simple$. There are exactly two topological types of one--sided simple closed curves~: bounding and non bounding.
We say that a one--sided simple closed curve is \emph{bounding} if its complement is orientable, otherwise we say that it is \emph{non bounding}. A simple closed curve is bounding if and only if it intersects any other one--sided simple closed curve. We denote by $\simple^-_b$ (resp. $\simple^-_{nb}$) the subset of bounding (resp. non bounding) elements of $\simple^-$, it forms an orbit under the mapping class group action. Let us note that a genus one nonorientable surface has only bounding one--sided simple closed curves, and an even genus nonorientable surface has only non bounding one--sided simple closed curves.

\part{Spaces of measured laminations}

\section{Nonorientable surfaces of small complexity}\label{sec:complexity}

 In this section we provide some examples that illustrate the results of the next sections. We quickly describe the projective space of measured laminations $\PML(S)$ in some particular cases. We consider the three compact nonorientable surfaces with Euler characteristic $-1$ and the three--holed projective plane.\par

The curve complex of these surfaces has been studied by Scharlemann (\cite{scharlemann}). We studied the nonorientable surfaces with Euler characteristic $-1$ in \cite{gendulphe}, and the three--holed projective plane in \cite{gendulphe-moduli}. All the statements made in this section can be found in these articles. We also refer to \cite{huang} for the particular case of the three--holed projective plane.\par

 The nonorientable surfaces with Euler characteristic $-1$ are very peculiar, and they appear as exceptions of some theorems. For instance, they do not admit pseudo--Anosov mapping classes, and they are hyperelliptic in the sense that there is a unique element in the mapping class group which acts as $-\mrm{Id}$ on the first homology group. This element is central and acts trivially on the space of measured laminations. In the sequel we denote by $N_{g,n}$ the compact nonorientable surface of genus $g$ with $n$ boundary components.

\subsection*{The two--holed projective plane}
This surface has only two simple closed geodesics (that are one--sided), and any other simple geodesic spirals around one of these closed geodesics or along a boundary component. So $\PML(N_{1,2})$ consists in only two points, that are exchanged by the mapping class group $\Mod(N_{1,2})$ which is isomorphic to $\Z/2\Z\times \Z/2\Z$.

\subsection*{The one--holed Klein bottle}
This surface has exactly one two--sided simple closed geodesic, that we call $\g_\infty$. It has infinitely many one--sided simple closed geodesics, that we denote by $(\g_n)_{n\in \Z}$. Any other simple geodesic spirals around a closed geodesic or along the boundary component. So $\PML(N_{2,1})$ is a circle with a marked point $\g_\infty$, which is the limit of the points $(\g_n)_{n\in \Z}$ as $n$ tends to $\pm \infty$. Each arc $[\g_n,\g_{,n+1}]\subset \PML(N_{2,1})$ consists in the set of projective measured laminations of the form $[t\g_n+(1-t)\g_{n+1}]$ with $t\in[0,1]$.\par

 The mapping class group $\Mod(N_{2,1})$ is isomorphic to $\mbf D_\infty\times\Z/2\Z$, where $\mbf D_\infty$ is the infinite dihedral group. The index two subgroup $\Z\leqslant \mbf D_\infty$ is generated by the Dehn twist along $\g_\infty$ and acts transitively on $\{\g_n~;~n\in\Z\}$.\par
 
 For any hyperbolic metric, the number of one--sided simple closed geodesics of length less than $L$ is asymptotically equivalent to a positive constant times $L$, the constant depending on the metric (\cite{gendulphe}).

\subsection*{The connected sum of three projective planes}
This surface is very close to be a one--holed torus. Indeed it has a unique simple closed geodesic $\g$ whose complement is a one--holed torus $\T$ (Scharlemann \cite{scharlemann}, see also \cite{gendulphe}). As a consequence there is a canonical embedding $\PML(\T)\subset \PML(N_3)$, and also a canonical isomorphism
$\Mod(N_3)\simeq \Mod^\ast(\T)$. We recall that the symplectic representation induces an isomorphism between $\Mod^\ast(\T)$ and $\mbf{GL}(2,\Z)$.\par
 We showed in \cite{gendulphe,komori} that any two--sided simple closed curve is homotopic to a curve contained in $\T$, and more generally any measured lamination $\l$ without one--sided closed leaf is contained in $\T$. The circle $\PML(\T)$ divides $\PML(N_3)$ into two open disks. One open disk consists in all projective measured laminations of the form $[\g+\l]$ with $\l\in\ML(\T)$. The other, given by the inequality $i(\g,\cdot)>0$, consists in all measured laminations with a one--sided closed leaf distinct from $\g$. We refer to \cite{komori} for more details and a nice picture of $\PML(N_3)$.\par 

  For any hyperbolic metric, the number of two--sided simple closed geodesics of length less than $L$ is asymptotically equivalent to a positive constant times $L^2$, the constant depending on the metric. The same is true for the number of one--sided simple closed geodesics of length less than $L$.

\subsection*{The three--holed projective plane}
The complement $N_{1,3}-\g$ 
of any one--sided simple closed geodesic $\g$ is a four--holed sphere. Let us denote by $\cone(\g)$ the set of measured laminations of the form $t\g+\l$ with $t>0$ and $\l\in\ML(N_{1,3}-\g)$. It is an open ball invariant under multiplication by a positive scalar, it projects onto an open disk in $\PML(N_{1,3})$ which is a $2$--sphere. Any other one--sided simple closed geodesic $\d$ intersects $\g$, therefore the balls $\cone(\g)$ are $\cone(\d)$ are disjoint. One easily observes that $\cone(\g)$ and $\cone(\d)$ are tangent (\emph{i.e.} their boundaries have  nontrivial intersection) if and only if $i(\g,\d)=1$.\par

 The projections on $\PML(N_{1,3})$ of the balls $\cone(\g)$ with $\g\in \simple^-$ form a packing of disks. There exists a homeomorphism $\f:\PML(N_{1,3})\rightarrow \Sph^2$ that sends this packing of disks on the so--called Apollonian packing (\cite{gendulphe-moduli}). Moreover $\f$ is equivariant with respect to the actions of $\Mod(N_3)$ and of a discrete subgroup $\G$ of $\Isom(\Hyp^3)\simeq\mbf{Conf}(\Sph^2)$. The complement of the packing of disks is a minimal $\Mod(N_{1,3})$--invariant closed subset of $\ML(N_{1,3})$. Huang and Norbury \cite{huang} showed that $\Mod(N_{1,3})\simeq(\Z/2\Z\oplus \Z/2\Z)\ast \Z/2\Z$.

\section{A decomposition of the space of measured laminations}\label{sec:decomposition}

\subsection*{The ball of a simple multicurve}
 Given a simple multicurve $\g=\g_1+\ldots+\g_k$, we call \emph{ball of $\g$} the set $\cone(\g)$ of measured laminations on $S$ that admit $\g_1,\ldots,\g_k$ as closed leaves:
\begin{eqnarray*}
\cone(\g) & = & \R_+^\ast \g_1+\ldots +\R_+^\ast \g_k +\ML(S-\g)\cup\{0\}.
\end{eqnarray*}
It is homeomorphic to an open ball of dimension
$\dim\cone(\g)=\dim\ML(S)-k^+,$
where $k^+$ is the number of two--sided components of $\g$. If all the components $\g_i$ are one--sided, then $\cone(\g)$ is an open subset of $\ML(S)$. The ball $\cone(\g)$ is convex in any train--track chart, and stable under multiplication by a positive scalar. Its projection on $\PML(S)$ is homeomorphic to an open ball of dimension $\dim\ML(S)-k^+-1$ bounded by a polyhedral sphere.\par
 
 Inside $\cone(\g)$, the Thurston measure $\muth^S$ decomposes as a product:
\begin{eqnarray}\label{eq:measure}
\muth^S & = & \diff \g_1 \otimes \ldots \otimes \diff\g_n\otimes \muth^{S-\g}.
\end{eqnarray}
We use abusively the notation $\diff \g_i$ for the differential of the weight of $\g_i$.
Let us consider $V=I_1\g_1+\ldots+I_n\g_n+U$ where $I_i\subset \R_+^\ast$ is an open interval and $U$ is an open subset of $\ML(S-\g)\cup\{0\}$. Open subsets like $V$ generate the Borel $\sigma$--algebra of $\ML(S)$. The number of integral points in $V$ is the product of the numbers of integral points in the $I_i$'s and in $U$. We conclude that \eqref{eq:measure} is true since the Thurston measure can be defined in terms of integral points.

The combinatorics of the intersection of balls is very simple:
$$\left\{\begin{array}{lc}
\cone(\g)\cap \cone(\delta)=\emptyset & \mathrm{if}\ i(\g,\delta)\neq 0,\\
\cone(\g)\cap \cone(\delta)=\cone(\g+\d) & \mathrm{if}\ i(\g,\delta)= 0.
\end{array}\right.$$
In particular, an intersection of balls is contractible whenever it is nonempty.

\subsection*{Laminations with a closed one--sided leaf}
We denote by $\ML^-(S)$ the set of measured laminations with a closed one--sided leaf, and by $\ML^+(S)$ the set of measured laminations without closed one--sided leaves. We have
\begin{eqnarray*}
\ML^-(S) & = & \bigcup_{\g\in\simple^-} \cone(\g).
\end{eqnarray*}
The collection of balls
$\cone(\g)$ such that $\g$ is a simple multicurve whose components are one--sided
forms an open cover of $\ML^-(S)$ which is stable under intersection. As each $\cone(\g)$ is contractible, the nerve of this cover has same homotopy type as $\ML^-(S)$. From the combinatorics of the intersection of balls we deduce that the nerve is the complex of one--sided closed curves, denoted by $\curve^-(S)$. A $n$--simplex of $\curve^-(S)$ is a family of $n+1$ disjoint isotopy classes of one--sided simple closed curves.\par

 Alternatively, one can show that $\curve^-(S)$ is a deformation retract of $\PML^-(S)$. Any $\l$ in $\ML^-(S)$ is uniquely written as $\l=\l^+ + \l^-$ where $\l^-$ is a weighted sum of one--sided closed curves, and $\l^+$ has no one--sided closed leaf. Let us define $H:[0,1]\times \ML^-(S)\rightarrow\ML^-(S)$ by $H(t,\l)=(1-t)\l^+ + \l^-$. This is clearly a homotopy between the identity of $\ML^-(S)$ and the retraction $\l\mapsto \l^-$. As $H$ commutes with the multiplication by a scalar, it induces a deformation retraction of $\PML^-(S)$ onto $\curve^-(S)$. Note that these two constructions are $\Mod(S)$-invariant.\par
 
Let us look at the connectivity of $\ML^-(S)$:
 
\begin{proposition}
The connected components of $\ML^-(S)$ are $\cup_{\g\in \simple^-_{nb}} \cone(\g)$
and the balls $\cone(\g)$ with $\g\in \simple^-_{b}$.
\end{proposition}

Each ball $\cone(\g)$ with $\g\in\simple^-_b$ is obviously a connected component of $\ML^-(S)$. So it remains to show that the subcomplex $\curve^-_{nb}(S)\subset\curve^-(S)$ spanned by $\simple_{nb}^-$ is connected. This comes directly from the following lemma.

\begin{lemma}
For any non disjoint $\g,\d\in\simple^-_{nb}$ we have $d(\g,\d)\leq 2 i(\g,\d)$ where $d$ is the distance on the $1$-skeleton of $\curve^-(S)$ obtained by fixing the length edge to $1$.
\end{lemma}

\begin{proof}
We endow $S$ with a hyperbolic metric, and we work with the geodesic representatives of the isotopy classes.\par
We proceed by induction on $i(\g,\d)$.
We first consider the case $i(\g,\d)=1$. 
Let $P$ be a sufficiently small tubular neighborhood of $\g\cup\d$. Then $P$ is a projective plane with two boundary components embedded in $S$.
The complement $S-P$ is a (possibly non connected) nonorientable surface. Otherwise $S$ would be of genus one, and any one--sided simple closed curve of $S$ would be bounding, which is impossible since $\g$ is non bounding. As $S-P$ is nonorientable, it contains a one--sided simple closed curve, necessarily disjoint form $\g$ and $\d$.
This shows that $d(\g,\d)=2$.\par 

Now we assume $i(\g,\d)\geq 2$. We cut $\g$, this produces a boundary component $\bar\g$ of length $2\ell(\g)$.
The trace of $\d$ in $S-\g$ consists in a collections of disjoint arcs $\d_1,\ldots,\d_m$ with $m=i(\g,\d)$.
We assume that a transverse orientation of $\d_1$ induces distinct orientations of $\bar\g$ at the endpoints of $\d_1$,
this is possible for $\d$ is one--sided. We denote by $p$ and $q$ the endpoints of $\d_1$ on $\bar\g$.
Let $\s$ be the shortest arc of $\g$ that joins $q$ to the antipodal point of $p$. This arc intersects $\d_1\cup\ldots \cup\d_m$ in at most $i(\g,\d)$ points, because its length is less than $\ell(\g)$.
From an arc parallel to $\d_1$ and a subegment of $\s$ one can make a one--sided curve $c$ with $i(c,\d_1)=0$ and $i(c,\d)<i(\g,\d)$. Moreover we have $i(c,\g)=1$. So $d(\g,c)=2$ and by induction $d(c,\d)\leq 2i(c,\d)\leq 2(i(\g,\d)-1)$. We conclude with the triangular inequality
\end{proof}

\subsection*{Genericity of laminations with a one--sided leaf}
The following theorem plays a crucial role in the proofs of our main theorems~:
 
\begin{theorem}[Danthony--Nogueira]\label{thm:nogueira}
$\ML^-(S)$ is a dense open subset of $\ML(S)$ of full Thurston measure.
\end{theorem}

We say that a subset of $\ML(S)$ has \emph{full Thurston measure} if its complement is $\muth$--negligible.
The measurable part of the statement is difficult, it involves the Rauzy induction for linear involutions (a generalization of interval exchanges). The topological part is relatively easy (Proposition~1.2 in \cite{scharlemann} is rather similar).

 In the following two corollaries we precise the structure of a generic lamination:
 
\begin{corollary}\label{cor:generic}
The set of measured laminations $\l\in\ML(S)$ such that $\l^-$ bounds an orientable subsurface is a dense open subset of full Thurston measure.
\end{corollary}

\begin{remark}
A collection of disjoint non isotopic simple closed curves bounds an orientable subsurface if and only if it is maximal among such collections. 
\end{remark}

\begin{proof}
We proceed by induction on the genus $g\geq 1$ of $S$.
 \emph{(Initialization)} If $g=1$, then the corollary comes directly from Theorem~\ref{thm:nogueira}.
 \emph{(Induction)} We assume the property true for all nonorientable surfaces of genus less than a given $g\geq 1$. From Theorem~\ref{thm:nogueira} it suffices to show that, for any one--sided simple closed curve $\g$, the set of $\l\in\cone(\g)$ such that $\l^-$ bounds an orientable subsurface is open and has full measure in $\cone(\g)$ (note that $\g$ is automatically a component of $\l^-$). If $S-\g$ is orientable, then the assertion is obvious.
If $S-\g$ is nonorientable, then we use the induction hypothesis and the decomposition of the Thurston measure \eqref{eq:measure}.
\end{proof}

From Corollary~\ref{cor:generic} and \cite[Lemma~2.4]{mirzakhani-imrn} we immediately get:

\begin{corollary}
Almost every measured lamination $\l\in\ML(S)$ is of the form $\l=\l^-+\l^+$ where $\l^-$ bounds an orientable subsurface, and $\l^+$ is a maximal measured lamination of $S-\l^-$.
\end{corollary}

\subsection*{An identity}
Here we present another consequence of Theorem~\ref{thm:nogueira} in the form of an identity. Following \cite{mirzakhani-annals} we set $X=(S,m)$ and, for any $L>0$,
\begin{eqnarray*}
B_X(L) & = & \{\l\in\ML(S)~;~\ell_m(\l)\leq L\},\\
b_X(L) & = & \muth(B_X(L)).
\end{eqnarray*}
For any simple closed geodesic $\g$ of $X$, we abusively denote by $X-\g$ the hyperbolic surface with boundary which is the metric completion of $X-\g$. In particular, $b_{X-\g}(1)$ is the Thurston measure of $\{\l\in\ML(S-\g)~;~\ell_{X-\g}(\l)\leq 1\}$ where we denote by $S-\g$ the underlying topological surface of $X-\g$.

\begin{proposition}
Let $X$ be a finite area nonorientable hyperbolic surface. Then
 \begin{eqnarray*}  
 b_X(1)   & =& \sum_{\g} \frac{(d-n)!}{d!}  \frac{b_{X-\g}(1)}{\ell_X(\g_1)\cdots \ell_X(\g_n)},
\end{eqnarray*} 
where $d=\dim\ML(S)$ and $\g=\g_1+\ldots+\g_n$ ($n\in\N^\ast$) runs over the set of maximal families of disjoint non isotopic one--sided simple closed curves. 
\end{proposition}

We don't have any application of this formula which shows a relation between the volume $b_{X-\g}(1)$ and the lengths $\ell_X(\g_1),\ldots,\ell_X(\g_n)$. It would be interesting to bound $b_{X-\g}(1)$ in terms of these lengths.

\begin{proof}
From Corollary~\ref{cor:generic} we have:
\begin{eqnarray*}
\muth(B_X(L)) & = & \sum_{\g} \muth(B_X(L) \cap \cone(\g))\quad (\forall L>0),
\end{eqnarray*}
where $\g=\g_1+\ldots+\g_n$ ($n\in\N^\ast$) runs over the set of maximal families disjoint non isotopic one--sided simple closed curves.
Using \eqref{eq:measure} we find:
\begin{eqnarray*}
\muth(B_X(L) \cap \cone(\g)) & = & \int_{L\geq \ell_X( x\cdot \g)} b_{X-\g}(L-\ell_X(x\cdot \g)) \diff x, \\
                        & = &b_{X-\g}(1)  \int_{L\geq \ell_X( x\cdot \g)} (L-x_1\ell_X(\g_1)-\ldots -x_n\ell_X(\g_n))^{d-n} \diff x ,\\                        
                        & = & \frac{b_{X-\g}(1)}{\ell_X(\g_1)\cdots \ell_X(\g_n)} \int_{L\geq y_1+\ldots+y_n} (L-y_1-\ldots -y_n)^{d-n}  \diff y, \\
                         & =&   \frac{b_{X-\g}(1)}{\ell_X(\g_1)\cdots \ell_X(\g_n)}  \frac{(d-n)!}{d!}  L^{d} ,
\end{eqnarray*}                        
where $d=\dim\ML(S)$ and $x\cdot\g=x_1\g_1+\ldots+ x_n\g_n$ for any $x\in(\R_+^\ast)^n$.
\end{proof}

\subsection*{A $\Mod(S)$--invariant continuous function}
We have seen that any measured lamination $\l\in\ML(S)$ is uniquely written as 
$\l =  \l^-+\l^+$ where $\l^+\in\ML^+(S)$ and $\l^-$ is a simple multicurve whose components are one--sided:
\begin{eqnarray*}
 \l^- & = &   a_1\g_1+\ldots+a_k\g_k
\end{eqnarray*}
where $a_1\geq\ldots\geq a_k>0$ and $\g_1,\ldots,\g_k$ are disjoint (non isotopic) one--sided simple closed curves.
We set $a_{k+1}=\ldots=a_g=0$ if $k$ is less than the genus $g$ of $S$.

 With the above notations, we define a function
$$\begin{array}{cccc}
w^-: & \ML(S) & \longrightarrow & \R_+^g \\
 & \l & \longmapsto & (a_1,\ldots, a_g)
\end{array}.$$

\begin{proposition}\label{pro:invariant-function}
The function $w^-$ is continuous and $\Mod(S)$--invariant. It induces a continuous and $\Mod(S)$--invariant function from $\PML^-(S)$ to $\proj(\R^g)$
\end{proposition}

\begin{remark}\label{rem:invariant-function}
The function $w^-$ is obviously nonconstant. The induced function from $\PML^-(S)$ to $\proj(\R^g)$ is nonconstant if $g\geq 2$.
\end{remark}

\begin{proof}
The only difficulty is the continuity, but $w^-$ gives in decreasing order the values of the weight functions of the one--sided simple closed geodesics, so the continuity of $w^-$ follows directly from the continuity of the weight functions (Lemma~\ref{lem:weight-function}).
\end{proof}
 
Given an isotopy class $\g$ of one--sided simple closed curves, the weight function $w_\g:\ML(S)\rightarrow\R$ is defined as follows: for any $\l\in\ML(S)$, the weight $w_\g(\l)$ is the unique nonnegative real number such that $\l=w_\g(\l)\cdot\g+\l'$ where $\l'\in\ML(S)$ has no leaf isotopic to $\g$. 
 
\begin{lemma}\label{lem:weight-function}
The weight function $w_\g$ of a one--sided curve $\g$ is continuous.
\end{lemma}

\begin{proof}
This function is identically equal to zero outside the open subset $\cone(\g)$. So it suffices to show that $w_\g$ is continuous on the closure of $\cone(\g)$ in $\ML(S)$.\par
 We fix a hyperbolic metric on $S$, and we denote by $C_r(\g)$ the collar of width $r>0$ around $\g$. As well--known, for $r>0$ sufficiently small, any simple geodesic disjoint from $\g$ either is asymptotic to $\g$, either does not penetrate $C_r(\g)$. In particular, a measured lamination in the closure of $\cone(\g)$ has no leaf that penetrate $C_r(\g)$ except $\g$. As a consequence, for a sufficiently small geodesic arc $k$ that intersects $\g$ transversely, we have $w_\g(\l)=i(k,\l)$ for any $\l$ in the closure of $\cone(\g)$. We conclude that $w_\g$ is continuous on the closure of $\cone(\g)$ by continuity of $i(\cdot,\cdot)$.
\end{proof}

\subsection*{Consequences on the dynamics of the $\Mod(S)$ action}

\begin{proposition}\label{pro:consequences}
The action of $\Mod(S)$ on $\PML(S)$ is not minimal, and is not topologically transitive if the genus of the surface is at least $2$.\par
The action of $\Mod(S)$ on $\ML(S)$ is not topologically transitive, in particular it is not ergodic with respect to the Thurston measure. 
\end{proposition}

\begin{remark}
In the case of an orientable surface, the action of the mapping class group on the space of measured laminations is ergodic with respect to the Thurston measure (\cite{masur}). Moreover, the orbit of any measured lamination without closed leaves is dense (see \cite{mirzakhani-lindenstrauss,hamenstadt}).
\end{remark}

\begin{proof}
The orbit of any $[\l]\in\PML^+(S)$ is disjoint from the open set $\PML^-(S)$, so the action of $\Mod(S)$ on $\PML(S)$ is not minimal. The other statements come from the existence of nonconstant, continuous and $\Mod(S)$--invariant functions (see Proposition~\ref{pro:invariant-function} and Remark~\ref{rem:invariant-function}).
\end{proof}

\section{Partial classification of orbit closures}\label{sec:orbit-closure}

In this section we study the topological dynamics of the $\Mod(S)$--actions on $\ML(S)$ and $\PML(S)$. We give a partial classification of its orbit closures. The idea is to use the $\Mod(S)$--invariance of the decomposition of $\ML^-(S)$ into open balls $\cone(\g)$ where $\g$ a simple multicurve whose component are one--sided.\par

 In the orientable case, Hamenst\"adt (\cite{hamenstadt}) and Lindenstrauss--Mirzakhani (\cite{mirzakhani-lindenstrauss}) gave a complete classification of the closed invariant subsets of the space of measured laminations. Their works rely on the properties of the Teichm\"uller geodesic and horocyclic  flows, for instance they both use the nondivergence of the Teichm\"uller horocyclic flow established by Minsky and Weiss (see \cite[\textsection 6]{mirzakhani-lindenstrauss} and \cite[Appendix]{hamenstadt}). It seems difficult to adapt their arguments since the Teichm\"uller horocylic flow is not well--defined in the nonorientable case (see \textsection\ref{sec:teichmuller}).\par

We recall that $S$ is a finite type nonorientable surface with $\chi(S)<0$. We denote by $\ML^+(S,\Q)$ the set of elements in $\ML(S,\Q)$ whose components are two--sided. The closure of $\ML^+(S,\Q)$ is included in $\ML^+(S)$, we prove that there is equality when $S$ has genus one (Lemma~\ref{lem:PML+PML}).

\subsection*{Orbit closures in $\ML(S)$}

Following \cite{mirzakhani-lindenstrauss} we define a \emph{complete pair} as a couple $\mc R=(R,\g)$ where $\g=x_1\g_1+\ldots+x_n\g_n$ ($x_i>0$) is a simple multicurve, and $R$ is a union of connected components of the complement $S-\g$. As in \cite{mirzakhani-lindenstrauss} we use the following notations: 
$$\mc G^{\mc R}  =  \g+(\ML^+(R)\cup\{0\})\quad \textnormal{and}\quad \mc G^{[\mc R]} =  \cup_{f\in\Mod(S)}\  \mc G^{f\cdot \mc R}.$$
To any measured lamination $\l$ we associate a complete pair $\mc R_\l=(R,\g)$ as follows: the multicurve $\g$ corresponds to the atomic part of the transverse measure of $\l$, and the subsurface $R$ is the union of the connected components of $S-\g$ that contain a noncompact leaf of $\l$.

\begin{theorem}\label{thm:closure-ML}
For any measured lamination $\l\in\ML(S)$ we have
$$\overline{\Mod(S)\cdot \l}\subset\mc G^{[\mc R_\l]}.$$
The inclusion is an equality if $R$ is orientable.
\end{theorem}

\begin{proof}
\emph{1) We show that $\overline{\Mod(S)\cdot \l}\subset\mc G^{[\mc R_\l]}$.}
 Let $\l_\infty$ be a measured lamination in the closure of $\Mod(S)\cdot\l$. We consider a sequence $(g_n)_n$ in $\Mod(S)$ such that $(g_n\l)_n$ converges to $\l_\infty$ in $\ML(S)$.\par
  
The measured lamination $\l$ is uniquely written in the form $\l=\g+\l'$ where $\g=x_1\g_1+\ldots+x_m\g_m$ ($x_i>0$) is a simple multicurve, and $\l'$ is a measured lamination without compact leaves. We clearly have $g_n\g\leq g_n\l$ and $g_n\l'\leq g_n\l$ for any $n$, here we consider a measured lamination as a geodesic current (see \textsection\ref{sec:geodesic-currents}). As $(g_n\l)_n$ converges, it comes that the sequences $(g_n\g)_n$ and $(g_n\l')_n$ are bounded in the space of Radon measures. Therefore they admit convergent subsequences with respect to the weak* topology (Tychonoff theorem). So we assume that $(g_n\g)_n$ and $(g_n\l')_n$ are convergent with respective limits $\g_\infty$ and $\l'_\infty$.\par

 Using the same argument, we assume that each sequence $(g_n\g_i)_n$ converges. Actually this implies that each sequence $(g_n\g_i)_n$ stablizes, for a compact subset of $\ML(S)$ contains only a finite number of simple closed geodesics. We fix an integer $N$ such that $g_n\g_i=g_N\g_i$ for any $n\geq N$ and any $i=1,\ldots,m$.\par
 
 Up to the choice of a subsequence and a bigger $N$, we assume that the mapping classes $g_N^{-1}g_n$ ($n\geq N$) do not permute the connected components of $S-\g_\infty$. Then the lamination $\l'_\infty=\lim_n g_n\l'$ is clearly contained in the subsurface $g_N R$, where $R$ is the subsurface defined before the statement of the theorem.\par
 
  It remains to show that $\l'_\infty$ has no closed one--sided leaf. This comes directly from the fact that the set $\ML+(S)$ of measured laminations without one--sided leaf is closed in $\ML(S)$, or equivalently that the set $\ML^-(S)$ of measured laminations with a one--sided leaf is open in $\ML(S)$ (see \textsection\ref{sec:decomposition}).
 
\emph{2) If $R$ is orientable.}
Then $\ML^+(R)=\ML(R)$ and, according to Theorem~1.2 of \cite{mirzakhani-lindenstrauss}, the orbit $\Mod(S-\g)\cdot \l'$ is dense in $\ML(R)$.
\end{proof}

\subsection*{Orbit closures in $\PML(S)$}
As well--known, for a finite type orientable surface with negative Euler characteristic, the action of the mapping class group on the projective space of measured laminations is minimal (see \cite{flp}). The following theorem shows that the mapping class group action has a very different topological dynamics in the nonorientable case.\par
Let $\mc{PG}^{[\mc R_\l]}$ denote the image of $\mc{G}^{[\mc R_\l]}$ in $\PML(S)$.

\begin{theorem}\label{thm:closures}
For any projective measured lamination $[\l]\in\PML(S)$ we have 
$$\overline{\Mod(S)\cdot [\l]}\subset \mc{PG}^{[\mc R_\l]}\cup\PML^+(S).$$
The inclusion is an equality if $S$ has genus one.
\end{theorem}

\begin{remark}
The theorem gives a complete classification of orbit closures in genus one.
\end{remark}

\begin{proof}[Proof of Theorem~\ref{thm:closures}]
The proof falls into two parts.\par
\emph{1) The inclusion.} Let $\l_\infty $ be a measured lamination such that $[\l_\infty]$ is in $\overline{\Mod(S)\cdot [\l]}$. We want to show that $[\l_\infty]\in \mc{PG}^{[\mc R_\l]}\cup\PML^+(S)$.
If $\l_\infty$ does not have a one--sided leaf, then $\l_\infty\in\ML^+(S)$ and we are done.
So we assume that $\l_\infty$ has a one--sided leaf. Then

\begin{lemma}\label{lem:stabilize}
There exists a number $\e>0$ and a sequence $(g_n)_n\subset \Mod(S)$ such that $(g_n\l)_n$ converges to $\e\l_\infty$ in $\ML(S)$.
\end{lemma}

\begin{proof}[Proof of Lemma~\ref{lem:stabilize}]
There exists two sequences $(\e_n)_n\subset \R_+^\ast$ and $(g_n)_n\subset \Mod(S)$ such that $\e_n(g_n\l)$ tends to $\l_\infty$ in $\ML(S)$ as $n$ tends to infinity. Clearly $\l_\infty$ belongs to the open ball $\cone(\l_\infty^-)$, therefore $g_n\l\in\cone(\l_\infty^-)$ for $n$ big enough. In particular, for $n$ big enough, any component of $\l^-_\infty$ is a component of $g_n\l^-$.\par
 Let us write $\l^-=x_1\g_1+\ldots+x_m\g_m$ where $\g_1,\ldots,\g_m$ are disjoint one--sided geodesics and $x_1,\ldots,x_m>0$. Up to the choice of a subsequence of $(g_n)_n$ we have $\l_\infty^-=y_1(g_n\g_1)+\ldots+y_l (g_n\g_m)$ for $n$ big enough and for some $y_1,\ldots,y_m\geq 0$  not all equal to zero and independent of $n$. The role of the subsequence is to avoid any permutation of the $\g_i$'s. From $\lim_n \e_n (g_n\l)=\l_\infty$ we deduce $\lim_n \e_n x_i=y_i$ and consequently $\lim_n \e_n=y_i/x_i$ for any $i=1,\ldots,m$. We conclude that $(g_n\l)_n$ converges to $\e\l_\infty $ with $\e=x_i/y_i$ ($i=1,\ldots,m$).
\end{proof}
 
 As a direct consequence of the above lemma we have $\e\l_\infty\in\overline{\Mod(S)\cdot\l}$.  Then Theorem~\ref{thm:closure-ML} implies $\e\l_\infty\in\mc G^{[\mc R_\l]}$ and $[\l_\infty]\in\mc{PG}^{[\mc R_\l]}$. This prove the first part of the statement.\par
 
 \emph{2) The equality.} From Lemmas~\ref{lem:PML+PML} and \ref{lem:PML+PML2} we have $\PML^+(S)\subset\overline{\Mod(S)\cdot[\l]}$. From the equality case of Theorem~\ref{thm:closure-ML} we have $\mc{PG}^{[\mc R_\l]}\subset\overline{\Mod(S)\cdot[\l]}$.
 This concludes the proof of the theorem.
\end{proof}

\begin{lemma}\label{lem:PML+PML}
If $S$ has genus one then $\PML^+(S)$ is the closure of $\PML^+(S,\Q)$.
\end{lemma}

\begin{proof}
Let $\l$ be a measured lamination without one--sided leaf: $\l\in\ML^+(S)$. We show that $\l$ is a limit of elements in $\ML^+(S,\Q)$. To do this we consider a sequence $(\g_n)_n\subset\ML(S,\Q)$ that converges to $\l$ in $\ML(S)$. If infinitely many $\g_n$ have no one--sided leaf, then we are done. So, up to the choice of a subsequence, we assume that each $\g_n$ has a one--sided leaf $\a_n$. Thus $\g_n$ belongs to the ball $\cone(\a_n)$. Our problem is to construct from $(\g_n)_n$ a sequence of \emph{two--sided} simple multicurves $(\d_n)_n\subset\ML^+(S,\Q)$ which has same limit as $\g_n$.\par
 Let $U\subset \R^m$ be a convex set which is a neighborhood of $\l$ in some train--track chart. For $n$ big enough we have $\g_n\in U$, and the segment $[\g_n,\l]$ is contained in $U$. It intersects the boundary $\partial\cone(\a_n)=\ML(S-\a_n)$ in a point $p_n$. This point corresponds to a lamination contained in $S-\a_n$. Note that $S-\a_n$ is a sphere with some punctures and holes, because $S$ has genus one. As $S-\a_n$ is \emph{orientable} there exists a \emph{two--sided} rational simple multicurve $\d_n$ on $S-\a_n$ which is at distance at most $1/n$ from $p_n$ in $U\subset \R^m$ with respect to the canonical Euclidean norm of $\R^m$. We have $\|\l-\d_n \|\leq \|\l-p_n\|+1/n\leq \|\l-\g_n \|+1/n$, and we conclude that the sequence of rational two--sided simple multicurves $(\d_n)_n$ converges to $\l$ in $\ML(S)$.
\end{proof}

\begin{lemma}\label{lem:PML+PML2}
For any $[\l]\in\PML(S)$ the orbit closure $\overline{\Mod(S)\cdot [\l]}$ contains $\PML^+(S,\Q)$, except if $[\l]$ is the projective class of the unique bounding one--sided simple closed geodesic of the connected sum of three projective planes.
\end{lemma}

\begin{remark}\label{rem:PML+PML2}
If $S$ is the projective plane with two boundary components then $\PML^+(S)$ is empty. If $S$ is the Klein bottle with one boundary component then $\PML^+(S)$ has only one point that represents the unique two--sided simple closed geodesic.
If $S$ is the connected sum of three projective planes, then the bounding one--sided simple closed geodesic is the unique geodesic lamination that does not intersect any two--sided simple closed geodesic.
\end{remark}

\begin{proof}
We assume that $S$ is not the projective plane with two holes or the Klein bottle with one hole (in both cases the lemma is trivially true). Let us consider a measured lamination $\l\in\ML(S)$ which is not the bounding geodesic of the connected sum of three real projective planes. Then there exists a two--sided simple closed geodesic $\g$ that intersects $\l$. Using the Dehn twist along $\g$ we find that $[\g]$ is in the closure of $\Mod(S)\cdot [\l]$ (see \cite[Proposition 3.4]{farb}).\par
 Given $\d\in \PML^+(S,\Q)$ there exists a two--sided simple closed geodesic $\g_1$ that intersects $\g$ and each leaf of $\d$. Using Dehn twists, we first show that $[\g_1]$ belongs to the closure of $\Mod(S)\cdot [\g]$, and then that $[\d]$ belongs to the closure of $\Mod(S)\cdot[\g_1]$. We conclude by transitivity of the relation \emph{to belong to the orbit closure of}.
\end{proof}

\subsection*{Topological transitivity in $\PML(S)$}
We recall that an action is \emph{topologically transitive} if it has a dense orbit.

\begin{proposition}\label{pro:transitivity}
The action of $\Mod(S)$ on $\PML(S)$ is topologically transitive if and only if $S$ has genus one.
\end{proposition}

\begin{remark}
A direct application of Theorem~\ref{thm:closures} shows that, if $S$ has genus one, then an orbit $\Mod(S)\cdot[\l]$ is dense in $\PML(S)$ if and only if $[\l]=[\g+\l']$ where $\g$ is a one--sided geodesic and $\l'$ has no closed leaf. 
\end{remark}

\begin{proof} We distinguish the two cases:\par
\emph{1) $S$ has genus one.}  We recall that $\ML^-(S)=\cup_\g \cone(\g)$ where $\g$ runs over the set of one--sided simple closed geodesics. From the hypothesis that $S$ has genus one we deduce that $\Mod(S)$ acts transitively on the set of components $\cone(\g)$ (there is only one topological type of one--sided simple closed geodesic). Using Lemma~\ref{lem:top-transitive} we conclude that the $\Mod(S)$--action on $\PML(S)$ is topologically transitive.\par

\emph{2) $S$ has genus at least two.} See Proposition~\ref{pro:consequences}.  
\end{proof}

\begin{lemma}\label{lem:top-transitive}
Let $\g$ be a simple closed geodesic such that $S-\g$ is orientable. Then the action of $\Mod(S)$ on the projective ball $\mc{ PB}(\g)$ is topologically transitive.
\end{lemma}

\begin{proof}
Let us consider the subset $\g+\ML(S-\g)\subset \cone(\g)$.
The action of $\Mod(S-\g)$ on $\ML(S-\g)$ is ergodic with respect to the Thurston measure on $\ML(S-\g)$ (Masur \cite{masur}), in particular it is topologically transitive (see also \cite{mirzakhani-lindenstrauss} for more precise results). We conclude as the projection $\ML(S)\rightarrow\PML(S)$ induces a $\Mod(S)$-invariant homeomorphism between $\g+\ML(S-\g)$ and $\mc{PB}(\g)$.
\end{proof}

\subsection*{Minimal invariant in $\PML(S)$}
We have studied in \textsection\ref{sec:complexity} the case of surfaces with $\chi(S)=-1$, so we assume $\chi(S)<-1$.
 Then the above lemma says that \emph{$\overline{\PML^+(S,\Q)}$ is the unique minimal $\Mod(S)$--invariant closed subset of $\PML(S)$}. This minimal invariant can alternatively be described in terms of pseudo--Anosov laminations (Papadopoulos and McCarthy (\cite{papadopoulos}). Indeed, the closure of the set of pseudo--Anosov laminations is $\Mod(S)$--invariant and contained in the closure of any $\Mod(S)$--orbit. The only problem is to show that this set is nonempty (this happens when $\chi(S)=-1$). Thurston explained how to construct pseudo--Anosov mapping classes provided that $S$ has a pair of two--sided simple closed geodesics that fill up $S$ (\cite[Theorem~7]{thurston}, see also \cite[Theorem~14.1]{farb}). Such a pair does not exist when $\chi(S)=-1$, but it does exist when $\chi(S)<-1$ (Lemma~\ref{lem:filling}).\par

\subsection*{Conjectures}
In view of our partial classifications it is natural to formulate the following conjectures:
 
\begin{conjecture}
We have $\PML^+(S)=\overline{\PML^+(S,\Q)}$.
\end{conjecture}

We have seen that this is true when $S$ has genus one (Lemma~\ref{lem:PML+PML}). It would be a first step in the direction of the following more general conjecture:

\begin{conjecture}
The inclusions in Theorems~\ref{thm:closure-ML} and \ref{thm:closures} are equalities, except if $\l$ is the bounding geodesic of the genus three closed surface.
\end{conjecture}

We have seen in Remark~\ref{rem:PML+PML2} that the bounding geodesic of the connected sum of three real projective planes is an exception.

\section{Growth of the number of simple closed geodesics}\label{sec:growth-simple}

We now prove Theorem~\ref{thm:1} on the growth of $\mcurve_{\g_0}=\Mod(S)\cdot\g_0$. Actually we are going to prove a more general result (Theorem~\ref{thm:growth-simple}). We first need to introduce some notations.\par

For any $N\in\N^\ast$, we denote by $\ML^-_N(S,\Z)$ the set of integral simple multicurves whose \emph{one--sided} components have weight at most $N$. In other words, a simple multicurve $\g=n_1\g_1+\ldots +n_k\g_k$ ($n_i\in\N^\ast$) belongs to $\ML^-_N(S,\Z)$ if and only if $n_i\leq N$ for every \emph{one--sided} component $\g_i$. This set is clearly $\Mod(S)$--invariant, and it contains infinitely many $\Mod(S)$--orbits since the weights of the two--sided components are not bounded. Note also that, for any $\g_0\in\ML(S,\Z)$, there exists $N$ such that $\mcurve_{\g_0}\subset \ML^-_N(S,\Z)$.\par

 The following theorem is the generalization of Theorem~\ref{thm:1}:
 
\begin{theorem}\label{thm:growth-simple}
Any nonorientable hyperbolic surface of finite area $(S,m)$ satisfies
\begin{eqnarray*}
\lim_{L\rightarrow\infty} \frac{\left\{\g\in\ML^-_N(S,\Z)~;~\ell_m(\g)\leq L \right\}}{L^{\dim\ML(S)}} =0.
\end{eqnarray*}
\end{theorem}

\begin{remark} In Part~\ref{part:currents} we extend this result to all multicurves (not necessarily simple) and to all filling geodesic currents (not necessarily Liouville currents of hyperbolic metrics).
\end{remark}

 The global scheme of the proof is similar to the one of \cite{mirzakhani-annals}: we introduce a family of counting measures $(\nu^L)_{L>0}$ on $\ML(S)$ and show its weak$^\ast$ convergence. We use the theorem of Danthony and Nogueira (\textsection\ref{sec:decomposition}) instead of Masur's ergodic theorem. We do not need to integrate over the moduli space, which is the most difficult part of Mirzakhani's proof (\cite{mirzakhani-annals}). 

\begin{proof}
 Let us consider the sequence of counting measures $(\nu^L)_{L>0}$ defined by
\begin{eqnarray*}
\nu^L  & = &  \frac{1}{L^{\dim\ML(S)}} \sum_{\g\in\ML^-_N} \mathbf{1}_{\frac{1}{L}\g},
\end{eqnarray*}
where for short we set $\ML^-_N=\ML^-_N(S,\Z)$. We have
$$\nu^L(B_m(1))  =   \frac{|B_m(1)\cap \frac{1}{L} \ML^-_N|}{L^{\dim\ML(S)}}
 =   \frac{|B_m(L)\cap \ML^-_N|}{L^{\dim\ML(S)}}
 = \frac{ | \{\g\in \ML^-_N~;~\ell_m(\g)\leq L\} |}{L^{\dim\ML(S)}}.$$
Thus the conclusion of the theorem is equivalent to
\begin{eqnarray*}
\lim_{L\rightarrow \infty} \nu^L(B_m(1)) &= & 0,
\end{eqnarray*}
which is a direct consequence of the weak$^\ast$ convergence of $(\nu^L)_L$ towards the zero measure (Proposition~\ref{pro:convergence-measure} and Lemma~\ref{lem:bdy-negligible}).
\end{proof}

\begin{proposition}\label{pro:convergence-measure}
The measure $\nu^L$ weak$^\ast$ converges to the zero measure as $L$ tends to infinity.
\end{proposition}

\begin{proof}
Let us recall that $(\mu^L)_L$ is the sequence of counting measures associated to the set of all integral points of $\ML(S)$ (see \eqref{eq:Thurston-measure} for an explicit formula). We have $\nu^L\leq \mu^L$ for any $L>0$, and we have seen (\textsection\ref{sec:preliminaries}) that $\mu^L$ weak$^\ast$ converges to $\muth$ as $L$ tends to infinity. We deduce that $(\nu^L)_{L>0}$ is relatively compact in the space of Radon measures (Tychonoff's theorem). Thus, to prove its convergence, it suffices to show that it has a unique limit point.\par
 Let $\nu^\infty$ be the weak$^\ast$ limit of a sequence $(\nu^{L_n})_n$ where $(L_n)_n$ is a sequence of positive real numbers that tends to infinity as $n$ tends to infinity. From the inequality $\nu^{L_n}\leq \mu^{L_n}$ we get that $\nu^\infty$ is absolutely continuous with respect to $\muth$. But its support $\mrm{supp}(\nu^\infty)$ is $\muth$--negligible (Lemma~\ref{lem:support-zero}), so we conclude that $\nu^\infty$ is the zero measure.
 \end{proof}

\begin{lemma}\label{lem:support-zero}
We have $\mrm{supp}( \nu^\infty)\subset \ML^+(S)$ which is $\muth$--negligible.
\end{lemma}

\begin{proof}
Let us consider a measured lamination $\l\in\ML^-(S)$. Our aim is to show that $\l\notin\mrm{supp}(\nu^\infty)$.
We write $\l=\l^-+\l^+$ with $\l^{+}\in\ML^+(S)\cup\{0\}$ and $\l^-=x_1\g_1+\ldots +x_k\g_k$ ($k\geq 1$) where $x_1,\ldots,x_k>0$ and $\g_1,\ldots,\g_k$ are disjoint (non isotopic) one--sided simple closed geodesics.\par

 We consider a neighborhood $U=I_1\g_1+\ldots+I_k\g_k+\ML^+(S-\l^-)\cup\{0\}$ of $\l$ where each $I_i\subset\R_+^\ast$ is a \emph{compact} interval containing $x_i$ in its interior. Clearly $\nu^\infty(U)=0$ implies $\l\notin\mrm{supp}(\nu^\infty)$. Therefore our aim is to show that $\nu^\infty(U)=0$.\par 
 
 We claim that $\nu^\infty(U)=\lim_{n}\nu^{L_n}(U)$. As $(\nu^{L_n})_n$ weak$^\ast$ converges to $\nu^\infty$ we just have to check that $\partial U$ is $\nu^\infty$--negligible. But $\nu^\infty$ is absolutely continuous with respect to $\muth$, and $\muth$ decomposes as a product on $U$ (see \textsection\ref{sec:decomposition}), hence $\partial U$ is $\nu^\infty$--negligible.\par 

  Now we show that $\nu^L(U)=0$ for $L$ big enough. 
 Let $\eta$ be an element of $\ML^-_N(S,\Z)\cap (L\cdot U)$ for some $L>0$. From the definition of $\ML^-_N(S,\Z)$ we have $\eta=n_1\g_1+\ldots+ n_k\g_k+\eta^+$ with $N\geq n_i\geq 0$ and $\eta^+\in\ML^+(S-\l^-)\cup\{0\}$. And from $\eta\in L\cdot U$ we have $n_i\in L\cdot I_i$. The conditions $N\geq n_i$ and $n_i\in L\cdot I_i$ imply that $\ML^-_N(S,\Z)\cap (L\cdot U)=\emptyset$ for $L$ big enough. This is precisely equivalent to $\nu^L(U)=0$ for $L$ big enough.
\end{proof}

Let us recall the following obvious lemma (see \cite[\textsection 4.2]{souto}):

\begin{lemma}\label{lem:bdy-negligible}
The boundary $\partial B_m(1)$ is $\muth$--negligible.
\end{lemma}

\begin{proof}
We have $B_m(1)\subset B_m(1+\e)-B_m(1-\e)$ for any $\e>0$ small enough.
By multiplicativity of the Thurston measure we find (for any $\e>0$ small enough)
\begin{eqnarray*}
\muth(\partial B_m(1)) & \leq & \muth(B_m(1+\e))-\muth(B_m(1-\e)),\\
\muth(\partial B_m(1)) & \leq & \muth(B_m(1)) \cdot ((1+\e)^d-(1-\e)^d),\\
\muth(\partial B_m(1)) & \leq & \muth(B_m(1)) \cdot \e\cdot P(\e), 
\end{eqnarray*}
where $d=\dim\ML(S)$ and $P(\e)$ is a polynomial in $\e$.
We conclude by taking the limit as $\e$ tends to zero.
\end{proof}

\section{Pairs of filling curves}

Let $S$ be a compact surface with $\chi(S)<0$. Following Thurston (\cite{thurston}) we say that a pair of simple closed geodesics $\{\g,\d\}$ \emph{fills up} $S$ if each component of $S-(\g\cup\d)$ is an open disk or a half--open annulus whose boundary lies in $\partial S$. Equivalently $\{\g,\d\}$ is filling if every simple closed geodesic of $S$ intersects $\g\cup\d$.

\begin{lemma}\label{lem:filling}
If $S$ is nonorientable, then it has a filling pair of two--sided simple closed geodesics if and only if $\chi(S)<-1$.
\end{lemma}

\begin{remark}
See \cite[Proposition~3.5]{farb} for a proof in the orientable case.
\end{remark}

\begin{proof}
If $\chi(S)=-1$ then one easily sees that $S$ has no pair of filling two--sided simple closed geodesics. So we assume $\chi(S)<-1$.\par
We consider a simple closed geodesic $\g$ whose complement is orientable. Let $\l$ be a \emph{filling} measured lamination of $S-\g$, that is a measured lamination that intersects any simple closed geodesic of $S-\g$. Let $\d$ be a two--sided simple closed geodesic of $S$ that intersects $\g$ (it exists thanks to the assumption $\chi(S)<-1$). The function $i(\d,\cdot)+i(\l,\cdot)$ is positive on $\ML(S)$.\par

 Let  $(\a_n)_n$ be a sequence of \emph{two--sided} simple closed geodesics that converges to $[\l]$ in $\PML(S)$. The existence of such a sequence is obvious for $\l$ is contained in an orientable subsurface. Let $(a_n)_n$ be a sequence of positive real numbers such that $(a_n\a_n)$ converges to $\l$ in $\ML(S)$.\par

 Let us show that \emph{there exists $N$ such that $\d$ and $\a_N$ fill up $S$}. By contradiction we assume that for each $n$ there exists $\b_n\in \simple$ such that $i(\d,\b_n)+i(\a_n,\b_n)=0$. Up to the choice of a subsequence, we assume that $(\b_n)_n$ converges in $\PML(S)$ to a point $[\b]$. Let $(b_n)_n$ be a sequence of positive real numbers such that $(b_n\b_n)$ converges to $\b$ in $\ML(S)$. By continuity of $i(\cdot,\cdot)$ on $\ML(S)\times \ML(S)$ it comes
$i(\d,\b)+i(\l,\b)=\lim_{n\rightarrow\infty} i(\d,b_n\b_n)+i(a_n\a_n,b_n\b_n)=0.$
This contradicts the positivity of $i(\d,\cdot)+i(\l,\cdot)$ on $\ML(S)$.
\end{proof}

\begin{lemma}
If $S$ is nonorientable, then it has a filling pair of one--sided simple closed geodesics.
\end{lemma}

\begin{proof}
We first assume $\chi(S)<-1$ and following the previous lemma we consider  $\{\a,\b\}$ a pair of filling two--sided simple closed geodesics. Let $(\a_n)_n$ and $(\b_n)_n$ be two sequences of one--sided simple closed geodesics that converge respectively to $\a$ and $\b$ in $\PML(S)$. One easily shows the existence of such sequences using Dehn twists. We choose two sequences $(a_n)_n$ and $(b_n)_n$ of positive real numbers such that $(a_n\a_n)$ and $(b_n\b_n)$ converge respectively to $\a$ and $\b$ in $\ML(S)$. We conclude by contradiction as in the proof of the previous lemma.\par
 For any nonorientable surface $S$ with $\chi(S)=-1$ it is not difficult to find an explicit filling pair of one--sided simple closed geodesics.
 \end{proof}

\part{Growth of geodesics with self--intersections}\label{part:currents}

The aim of this part is to establish the following theorem: 

\begin{theorem}
Let $(S,m)$ be a nonorientable hyperbolic surface of finite area (possibly with geodesic boundary). For any integral multicurve $\g_0$ we have
\begin{eqnarray*}
\lim_{L\rightarrow \infty} \frac{\left|\left\{\g\in\mcurve_{\g_0}~;~\ell_m(\g)\leq L  \right\}   \right|}{L^{\dim\ML(S)}} & = & 0.
\end{eqnarray*} 
\end{theorem}

\begin{remark}
 Actually we prove a more general result (Theorem~\ref{thm:final}) that deals with geodesic currents.
\end{remark}

 This result contrasts with the result of Mirzakhani (\cite{mirzakhani-preprint}) which states that, for orientable surfaces, the above limit exists and is positive. It contrasts also with the fact that for any $k\geq 1$ there exists $C_k>0$ such that
$$\frac{1}{C_k}\cdot L^{\dim\ML(S)} \leq |\{\g\in \mcurve_k~;~\ell_m(\g)\leq L\}| \leq C_k \cdot L^{\dim\ML(S)},$$
for any $L$ large enough. Here we denote by $\mcurve_k$ the set of integral multicurves with exactly $k$ self--intersections. We refer to \cite[Corollary~3.6]{souto} for a proof of this fact that works for both orientable and nonorientable surfaces of finite type. This kind of estimates was first obtained by J.~Sapir (\cite{sapir,sapir-imrn}).\par

 As in the case of simple closed curves, the theorem comes with the convergence of a family of counting measures $(\nu_{\g_0}^L)_{L>0}$ defined by
 \begin{eqnarray*}
 \nu_{\g_0}^L & = & \frac{1}{L^{\dim\ML(S)}}\ \sum_{\g\in\mcurve_{\g_0}} \mathbf{1}_{\frac{1}{L}\g}.
 \end{eqnarray*}
These are $\Mod(S)$--invariant measures over the space of geodesic currents $\current(S)$ (see \textsection\ref{sec:geodesic-currents}). We show that the accumulation points of $\mcurve_{\g_0}$ in $\current(S)$ are contained in $\ML^+(S)$ (\textsection\ref{sec:accumulation-currents}). This implies that the support of any limit point of $(\nu^L_{\g_0})_L$ is contained in $\ML^+(S)$. Then we conclude that $(\nu_{\g_0}^L)_L$ converges to the zero measure (\textsection\ref{sec:convergence-currents}) using a theorem of Erlandsson and Souto (\cite{souto}).\par

 Along the way we show that, if $(\g_n)_n$ is a sequence of closed geodesics that converges to a one--sided simple closed geodesic in the projective space of geodesic currents $\pcurrent(S)$, then either $(\g_n)_n$ stabilizes either the number of self--intersections of $\g_n$ tends to infinity with $n$ (Proposition~\ref{pro:accumulation-current}).

\section{Geodesic currents}\label{sec:geodesic-currents}

In this section we recall some basic facts about Bonahon's geodesic currents. We refer to the article \cite{bonahon-inventiones}, or to the textbook \cite{martelli}, for more details and complete proofs. For sake of simplicity, we restrict our attention to closed surfaces, but it is explained in \cite[\textsection 4.1]{souto} how to deal with punctured surfaces.

\subsection*{Space of geodesics}
 Let $(S,m)$ be a closed hyperbolic surface. We denote by $\tilde S_\infty$ the boundary at infinity of the universal cover $\tilde S$. A geodesic of $(\tilde S,\tilde m)$ is encoded by its endpoints in $\tilde S_\infty$, that is why we call \emph{space of geodesics} of $\tilde S$ the quotient
$\Geod(\tilde S)=(\tilde S_\infty\times\tilde S_\infty-\Delta)/(\Z/2\Z)$
where $\Delta$ is the diagonal of $\tilde S_\infty\times\tilde S_\infty$ and $\Z/2\Z$ acts by transposition. \par
 
 Given another hyperbolic metric $m'$ on $S$, the identity $id_{\tilde S}:(\tilde S,\tilde m)\rightarrow (\tilde S,\tilde m')$ is a quasi--isometry, therefore it extends to a $\pi_1(S)$--invariant homeomorphism between the visual boundaries. This shows that $\tilde S_\infty$ and $\Geod(\tilde S)$ are of topological nature, they do not depend on $m$. Similarly, the objects introduced below are topological and independent of $m$ (except $\l_m$).\par
 
\subsection*{Geodesic currents}
A \emph{geodesic current on $S$} is a $\pi_1(S)$--invariant (positive) Radon measure on $\Geod(\tilde S)$. We denote by $\current(S)$ the space of geodesic currents on $S$ endowed with the weak$^\ast$ topology. It is stable under addition and multiplication by a positive scalar. The projective space of geodesic currents $\pcurrent(S)$ is compact with respect to the quotient topology (\cite[Proposition~5]{bonahon-inventiones}).\par

\subsection*{Multicurves as geodesic currents}
 To a (primitive) closed geodesic $\g$ of $(S,m)$ corresponds the geodesic current $\sum_{\tilde \g} \bf 1_{\tilde \g}$ where $\tilde \g$ runs over the set of lifts of $\g$. By linearity, this map extends to a canonical injection of the set of integral multicurves $\mcurve$ into the space of geodesic currents $\current(S)$.\par
 The set of \emph{integral} multicurves is discrete in $\current(S)$, but the set of closed geodesics is dense in $\pcurrent(S)$ (\cite[Proposition~2]{bonahon-inventiones}). The geometric intersection between closed geodesics extends to a bilinear continuous function $i:\current(S)\times \current(S)\rightarrow\R_+$. 

\subsection*{Hyperbolic metrics as geodesic currents}
  Any hyperbolic metric $m$ on $S$ determines a geodesic current $\l_m$, called the \emph{Liouille current} of $m$. As indicated by its name, it is related to the Liouville form on $T^\ast S$. For our purpose, it is enough to mention the following properties: the map $m\mapsto \l_m$ induces a topological embedding $\teich(S)\rightarrow\current(S)$ and, for any multicurve $\g\in\mcurve$, we have the following nice relation $i(\l_m,\g)=\ell_m(\g)$.

\subsection*{Measured laminations}
The injective map $\ML(S;\Q)\subset\mcurve\hookrightarrow\current(S)$ extends to a topological embedding of $\ML(S)$ into $\current(S)$. The image of $\ML(S)$ coincide with the light cone $\{c\in\current(S)~;~i(c,c)=0\}$.

\section{Accumulation points of $\mcurve_{\g_0}$ in $\current(S)$}\label{sec:accumulation-currents}

\begin{proposition}\label{pro:accumulation-currents}
For any $\g_0\in\mcurve$, the accumulation points of $\mcurve_{\g_0}$ in $\pcurrent(S)$ are contained $\PML^+(S)$.
\end{proposition}

\begin{remark}
One easily shows that any point in $\overline{\PML^+(S,\Q)}$ is an accumulation point of $\mcurve_{\g_0}$ in $\current(S)$, as well as any projective pseudo--Anosov lamination.
\end{remark}

\begin{proof}
Combine the Lemma~\ref{lem:limitinthecone} with Corollary~\ref{cor:accumulation-current}.
\end{proof}

\begin{lemma}\label{lem:limitinthecone}
The accumulation points of $\mcurve_{k}$ in $\pcurrent(S)$ are contained in $\PML(S)$.
\end{lemma}

\begin{proof}
Let $[\l]$ be an accumulation point of $\mcurve_k$ in $\pcurrent(S)$. There exist a sequence $(\d_n)_n$ of \emph{distinct} elements of $\mcurve_k$, and a sequence $(d_n)_n$ of positive real numbers such that $(d_n\d_n)_n$ converges to $\l$ in $\current(S)$. The geodesic current $\l$ is a measured lamination if and only if $i(\l,\l)=0$ (\cite[Proposition~17]{bonahon-inventiones}).\par
 We fix a hyperbolic metric $m$ on $S$.
On one hand we have $\ell_m(\d_n)\rightarrow \infty$ as $n$ tends to infinity (the $\d_n$ are distinct).
On another hand $\ell_m(d_n \d_n)\rightarrow\ell_m(\l)$ as $n$ tends to infinity ($\ell_m$ is continuous). So $d_n\rightarrow 0$ as $n$ tends to infinity. We conclude that 
$i(\l,\l) = \lim_n i(d_n\d_n ,d_n\d_n ) =k \lim_n d_n^2=0$ by continuity of $i(\cdot,\cdot)$.
 \end{proof}

\subsection*{Neighborhood of a lamination with a one--sided leaf}

\begin{proposition}\label{pro:accumulation-current}
Let $\l\in\ML^-(S)$ be a measured lamination with a one--sided closed leaf $\g$. For any $k\geq 1$, there exists a neighborhood $U_k$ of $\l$ in $\current(S)$ such that for any geodesic current $c\in U_k$ there exists a geodesic $\d\in \supp(c)$ that projects either on $\g$ either on a geodesic with at least $k$ self--intersections.
\end{proposition}

\begin{remark}
\begin{enumerate}
\item We do not say that the projection of $\d$ on $S$ is closed.
\item The property satisfied by $U_k$ is invariant by multiplication by a scalar. As a consequence the cone $\R_+^\ast U_k$ satisfies this property.
\end{enumerate}
\end{remark}

\begin{corollary}\label{cor:accumulation-current}
Let $[\l]\in \PML(S)$ be an accumulation point of $\mcurve_{\g_0}$ in $\pcurrent(S)$. Then $\l$ has no one--sided leaf.
\end{corollary}

\begin{proof}[Proof of Corollary~\ref{cor:accumulation-current}]
Let $[\l]\in\PML(S)$ be the limit in $\pcurrent(S)$ of a sequence $(\d_n)_n$ of elements of $\mcurve_{\g_0}$. There exists a sequence $(d_n)_n$ of positive real numbers such that $(d_n\d_n)_n$ converges to $\l$ in $\current(S)$. As noted in the proof of Lemma~\ref{lem:limitinthecone} we have $\lim_n d_n=0$.\par
 
 By contradiction we assume that $\l$ has a one--sided leaf $\g$. We denote by $U$ the neighborhood of $\l$ given by Proposition~\ref{pro:accumulation-current} for $k=i(\g_0,\g_0)+1$. For $n$ big enough we have  $d_n\d_n\in U$, which implies that $\g$ is a component of $\d_n$, and consequently can be written in the form $\d_n= c_n\g+\d'_n$ with $c_n\in\N^\ast$, and $\d_n'\in\mcurve_{\leq i(\g_0,\g_0)}$ such that $\g$ is not a component of $\d'$. \par
 
  The weight $c_n$ is uniformly bounded by the number of components of $\g_0$ counted with multiplicity. Therefore $\lim_n d_nc_n=0$ and $\lim_n d_n\d'_n=\lim_n d_n\d_n=\l$. So $d_n \d'_n$ belongs to $U$ for $n$ big enough. This contradicts Proposition~\ref{pro:accumulation-current}.
\end{proof}

\begin{proof}[Proof of Proposition~\ref{pro:accumulation-current}]
We fix a hyperbolic metric $m$ on $S$. Let $C$ be a collar neighborhood of $\g$ which is homeomorphic to a M\"obius band. We denote by $\tilde \g$ a lift of $\g$ to the universal cover, and by $\tilde C$ the lift of $C$ that contains $\tilde \g$.\par

We first choose $U_k$. According to Lemma~\ref{lem:intersection-collar} there exists $L_k>0$ such that any geodesic arc $\a$ in $C$ of length $\ell_m(\a)\geq L_k$ has at least $k$ self--intersections. We take $U_k$ to be the neighborhood of $\l$ given by Lemma~\ref{lem:collar-lift} for $L=L_k$.\par

 Let us show that $U_k$ satisfies the expected property. For any $c\in U_k$ there is a geodesic $\d\in \supp(c)$ such that $\d\cap \tilde C$ has length at least $L_k$ (Lemma~\ref{lem:collar-lift}). If $\d$ is asymptotic to $\tilde \g$, then we conclude that $\tilde \g\in \supp(c)$ because $\supp(c)$ is a closed $\pi_1(S)$--invariant subset of $G(\tilde S)$. If $\d$ is not asymptotic to $\tilde \g$, then $\d\cap \tilde C$ is a closed segment that projects onto a geodesic arc $\a$ in $C$ of length $\ell_m(\a)\geq L_k$. In that case we conclude that the projection of $\d$ on $S$ has at least $k$ self--intersections (by choice of $L_k$ and $U_k$).
\end{proof}

With the notations introduced in the proof of Proposition~\ref{pro:accumulation-current} we have:

\begin{lemma}\label{lem:collar-lift}
For any $L>0$ there exists a neighborhood $U\subset \current(S)$ of $\l$ such that for any $c\in U$ there is a geodesic $\d\in\supp(c)$ such that $\d\cap\tilde C$ has length at least $L$.
\end{lemma}

\begin{remark}
In the proof of the lemma, we only use the fact that $\l$ has a closed leaf $\g$, we do not use the hypothesis $\g$ one--sided.
\end{remark}

\begin{proof}
We denote by $w$ the width of $C$. Let $p,q$ be two points on $\tilde \g$ at distance $L+2w$ from each other. We denote by $V\subset\mc G(S)$ the set of geodesics that intersect the balls of radius $w>0$ centered at $p$ and $q$. This is an open subset of $\mc G(\tilde S)$. We denote by $U$ the set of geodesic currents $c\in \current(S)$ such that $c(V)>0$, or equivalently $U=\{c\in \current(S)~;~\supp(c)\cap V\neq \emptyset\}$. It is an open neighborhood of $\l$.\par
 By construction, for any $c\in U$ there is a geodesic $\d\in \supp(c)$ that intersects the balls of radius $w$ centered at $p$ and $q$, thus $\d\cap\tilde C$ has length at least $L$. 
\end{proof}

\begin{lemma}\label{lem:intersection-collar}
Let $\g$ be a one--sided simple closed geodesic of a hyperbolic surface, and $C$ be a collar neighborhood of $\g$ homeomorphic to a M\"obius band. Any geodesic arc $\a$ in $C$ satisfies
$$\left| i(\a,\a)-\frac{\ell(\a)}{2\ell(\g)} \right| \leq \frac{\ell(\partial C)+2w}{2\ell(\g)}+1,$$
where $i(\a,\a)$ is the number of self--intersections of $\a$, and $w$ is the width of $C$. 
\end{lemma}

This lemma says that, for a geodesic arc $\a$ contained in a collar neighborhood of a one--sided geodesic, the length $\ell(\a)$ is quasi--proportional to the intersection number $i(\a,\a)$.

\begin{proof}
We fix two distinct points $p$ and $q$ on $\partial C$. We first enumerate the geodesic arcs whose endpoints are $p$ and $q$.
Let $\b$ be a loop based at $p$ whose homotopy class generates $\pi_1(C,p)\simeq\Z$, and let $\d$ be an arc of $\partial C$ that goes from $p$ to $q$. We assume that the orientation of $\b$ and $\d$ are compatible (note that $\b^2$ is homotopic to $\partial C$). The collection $\{\d\ast\b^k\}_{k\in \Z}$ is a set of representatives of the homotopy classes of paths from $p$ to $q$. We denote by $\a_k$ the unique \emph{geodesic} arc homotopic to $\d\ast \b^k$  (Figure~\ref{fig:arcs-mobius}).\par

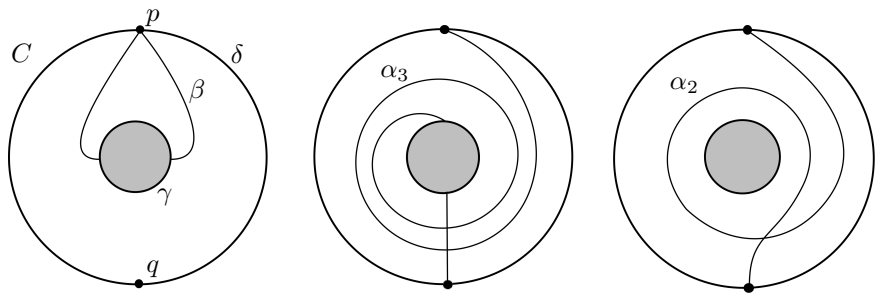
\begin{figure}[h]
\centering
\labellist
\pinlabel $C$ at 17 300
\pinlabel $p$ at 180 345
\pinlabel $q$ at 180 30
\pinlabel $\g$ at 195 120
\pinlabel $\b$ at 235 252
\pinlabel $\d$ at 285 300
\pinlabel $\a_3$ at 481 275
\pinlabel $\a_2$ at 841 260
\endlabellist
\includegraphics[scale=0.3]{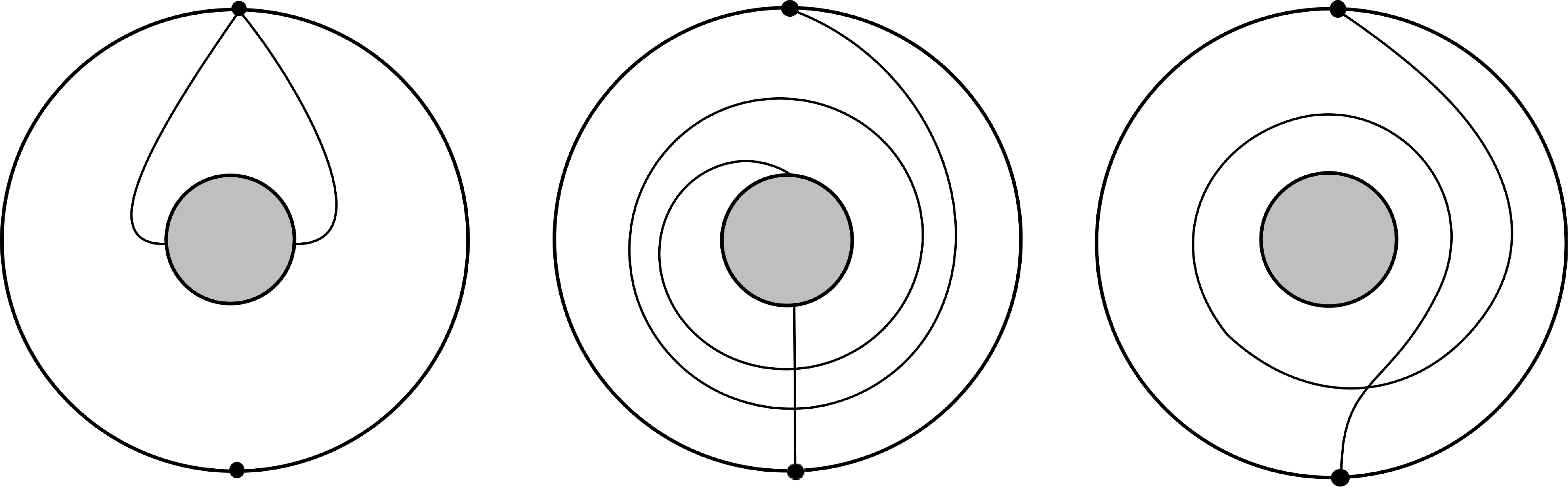}
\caption{Examples of geodesic arcs $\a_k$}\label{fig:arcs-mobius}

\end{figure}

 Now we show that $\left|i(\a_k,\a_k)-\frac{|k|}{2}\right| \leq 1$.
 The universal cover $\tilde C$ of $C$ is an infinite band isometric to the $w$--tubular neighborhood of a geodesic of the hyperbolic plane. The Deck transformation given by $\b$ acts by translation--reflection along the geodesic. A lift $\tilde \a_k$ of $\a_k$ divides $\tilde C$ into two connected components. The number $i(\a_k,\a_k)$ is half the number of lifts of $\a_k$ whose endpoints do not belong to the same component of $\tilde C-\tilde \a_k$ (Figure~\ref{fig:cover}). One easily finds that 
 $$i(\a_k,\a_k)=\left\{ \begin{array}{lll} k/2 & \textnormal{if} & k\geq0 \textnormal{ is even} \\
 |k|/2-1 & \textnormal{if} & k<0 \textnormal{ is even} \\
  (k+1)/2 & \textnormal{if} & k>0 \textnormal{ is odd} \\
  (|k|-1)/2 & \textnormal{if} & k<0 \textnormal{ is odd}  \end{array}\right. .$$
In the case $p=q$ we find that $i(\b^k,\b^k)$ is the integer part of $|k|/2$.

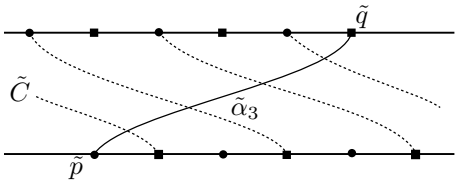
\begin{figure}[h]
\centering
\labellist
\pinlabel $\tilde C$ at 20 89
\pinlabel $\tilde p$ at  90 -10
\pinlabel $\tilde q$ at  445 180
\pinlabel $\tilde \a_3$ at  300  65
\endlabellist
\includegraphics[scale=0.3]{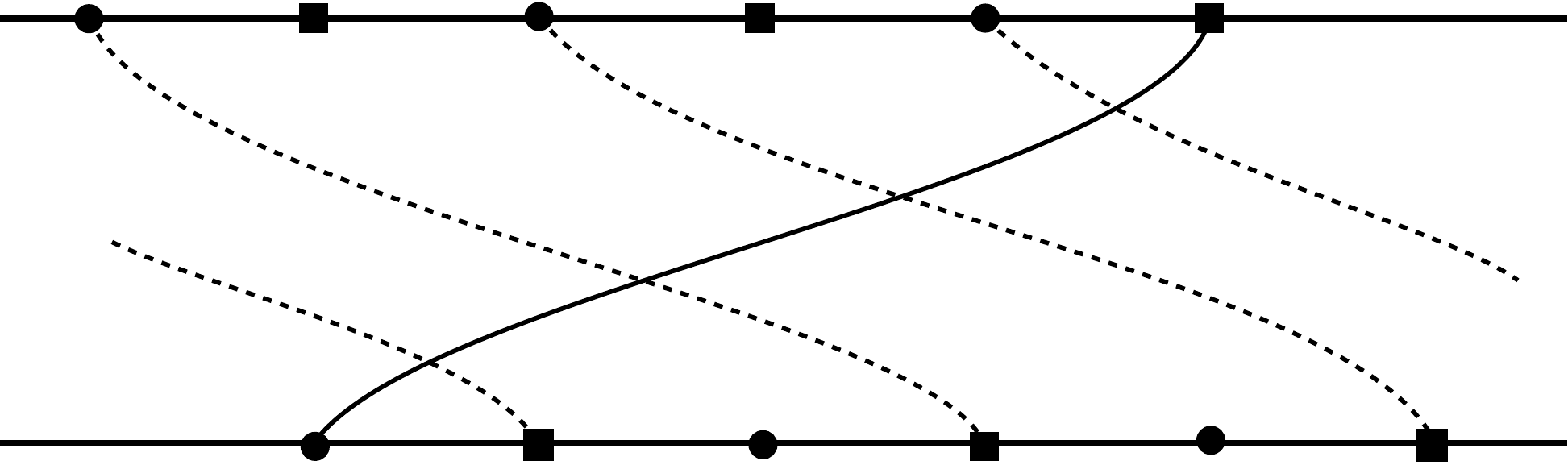}
\caption{The lift $\tilde \a_3$ in the universal cover $\tilde C$}\label{fig:cover}
\end{figure}

 By comparing the lengths of $\a_k$ and $\g^k$ with the lengths of some piecewise geodesic loops in the same homotopy classes we get
$\ell(\g^k)+\ell(\d)+2w>\ell(\a_k)$ and $\ell(\a_k)+\ell(\d) > \ell(\g^k)$,
where $w$ is the width of the collar $C$. This implies
$$\frac{\ell(\a_k)+\ell(\d)}{\ell(\g)}  >|k|> \frac{\ell(\a_k)-\ell(\d)-2w}{\ell(\g)}.$$
We conclude using the above expression of $i(\a_k,\a_k)$.
\end{proof}

\section{Convergence of counting measures}\label{sec:convergence-currents}

Let $\g_0\in\mcurve$ be an integral multicurve. For any $L>0$ we set
\begin{eqnarray*}
\nu_{\g_0}^L & = &\frac{1}{L^{\dim \ML(S)}}\ \sum_{\g\in \mcurve_{\g_0}} \mathbf{1}_{\frac{1}{L}\g}.
\end{eqnarray*}
This is a locally finite $\Mod(S)$--invariant Borel measure on the space $\current(S)$.

\begin{theorem}\label{thm:convergence-measures}
The measure $\nu^L_{\g_0}$ tends to the zero measure (with respect to the weak$^\ast$ topology) as $L$ tends to infinity.
\end{theorem}

\begin{proof}
According to the Proposition~\ref{pro:souto} of Erlandsson and Souto it suffices to show that the only limit point of $(\nu_{\g_0}^L)_{L>0}$ is the zero measure. Such limit point is absolutely continuous with respect to $\muth$ (same proposition of Erlandsson and Souto) and supported on $\ML^+(S)$ (Lemma~\ref{lem:support-current}). The conclusion comes from the fact that $\ML^+(S)$ is $\muth$--negligible (Theorem~\ref{thm:nogueira} of Danthony and  Nogueira).
\end{proof}

\begin{lemma}\label{lem:support-current}
Any limit point of the family $(\nu^L_{\g_0})_{L>0}$ is supported on $\ML^+(S)$.
\end{lemma}

\begin{proof}
Let $(L_n)_{n}$ be a sequence of positive numbers that converges to infinity, and let us assume that $(\nu^{L_{n}}_{\g_0})_{n}$ converges to a measure $\mu$ in the weak$^\ast$ topology. For any open set $U$ we have $\liminf_n \nu_{\g_0}^{L_n}(U)\geq \mu(U)$, which implies
$$\left|\frac{1}{L_n}\cdot\mcurve_{\g_0}\cap U\right|\geq \frac{\mu(U)}{2}L_n^{\dim\ML(S)}$$
for any $n$ big enough. We deduce that if $U\cap \supp(\mu)\neq \emptyset$ then the projection of $U$ in $\pcurrent(S)$ contains infinitely many points of  the projection of $\mcurve_{\g_0}$. It follows that the projection of any point $c\in\supp(\mu)$ is an accumulation point of $\mcurve_{\g_0}$ in $\pcurrent(S)$. We conclude that $c\in\ML^+(S)$ by applying Proposition~\ref{pro:accumulation-currents}.
\end{proof}

The proposition below is the nonorientable analogue of \cite[Proposition~4.1]{souto}, 
it has the same proof until the use of Masur ergodic theorem (as mentioned in the remark below).

\begin{proposition}[Erlandsson--Souto]\label{pro:souto}
Let $(L_n)_{n}$ be a sequence of positive real numbers that tends to infinity with $n$. There exists a subsequence $(L_{n_i})_i$ such that the sequence of measures $(\nu^{L_{n_i}}_{\g_0})_{i}$ converges in the weak$^\ast$ topology to a measure which is absolutely continuous with respect to the Thurston measure $\muth$.
\end{proposition}

\begin{remark}
We recall that, if $S$ is orientable, then the action of $\Mod(S)$ on $\ML(S)$ is ergodic with respect to $\muth$ (Masur \cite{masur}). This explains the difference between Theorem~\ref{thm:convergence-measures} and \cite[Proposition~4.1]{souto}.   
\end{remark}

 The proof of \cite[Proposition~4.1]{souto} involves all the preceding results of the article. It might not be clear that they all apply to nonorientable surfaces. Let us say few words about it to convince the reader.\par
  Some results extend directly to nonorientable surfaces by lifting the situation to the orientation cover (see for instance \cite[Proposition~2.1]{souto}), but in general this does not work. The main problem occurs when looking at some limit of the form $\lim_{L\rightarrow\infty} \frac{1}{L^{\dim S}} |\{\g\in\ast ~;~\ell_m(\g)\leq L \}|$ (see Corollary~3.6, Theorem~1.2 and Lemma~3.5 in \cite{souto}). Indeed the quantitiy $\dim\ML(S)$ looses its meaning when we pass to the orientation cover (it does not represent anymore the dimension of the space of measured laminations of the surface we are dealing with). Still the results concerned extend to the nonorientable setting. The idea behind them is to count multicurves using some specific graphs (train--tracks or a generalization of them called \emph{radalla}). To a multicurve $\g$ immersed on such a graph $\tau$ we associate the vector $\omega_\g\in \N^{E(\tau)}$ that gives the number of times each edge of the graph is followed. This correspondance is bounded--to--one (\cite[Proposition~2.1]{souto}) so that counting multicurves is almost counting the vectors in $\N^{E(\tau)}$ that satisfy some integral equations (the \emph{switch equations}). These equations define a linear subspace whose dimension is exactly $\dim\ML(S)$. This explains that the number of vectors $\omega_\g$ of norm less than $L$ is bounded by a constant times $L^{\dim\ML(S)}$. This kind of argument works perfectly well for nonorientable surfaces.\par
   
\section{Proof of the theorem}
  
 We prove the following theorem which is more general than Theorem~\ref{thm:2}: 
 
\begin{theorem}\label{thm:final}
Let $S$ be a nonorientable surface of finite type with $\chi(S)<0$. For any $\g_0\in\mcurve$, and for any filling geodesic current $c\in\current(S)$, we have
\begin{eqnarray*}
\lim_{L\rightarrow \infty} \frac{\left| \left\{ \g\in\mcurve_{\g_0}~;~i(c,\g)\leq L \right\} \right|}{L^{\dim\ML(S)}} & = & 0.
\end{eqnarray*}
\end{theorem}

\begin{remark}
It is the nonorientable analogue of \cite[Proposition~4.3]{souto}.
\end{remark}

We say that a geodesic current \emph{fills} the surface if every geodesic of $\tilde S$ intersects transversely a geodesic in its support. The Liouville current of a hyperbolic metric has full support, therefore it is a filling geodesic current. For our purpose, the important property is the following: if $c$ is a filling geodesic current, then the unit ball $B_c(1)=\{\l\in\current(S)~;~i(c,\l)\leq 1\}$ is compact (see \cite[Proposition~8.2.25]{martelli}).

\begin{proof}
We have 
$$\lim_{L\rightarrow \infty} \frac{\left| \left\{ \g\in\mcurve_{\g_0}~;~i(c,\g)\leq L \right\} \right|}{L^{\dim\ML(S)}}  = \lim_{L\rightarrow \infty} \nu_{\g_0}^L(B_c(1))$$
We conclude by Theorem~\ref{thm:convergence-measures} and by compacity of $B_c(1)$.
\end{proof}

\part{Geometry of Teichm\"uller spaces}

In this last part, we study the Teichm\"uller and Weil--Petersson geometry of Teichm\"uller spaces. We focus on the volume of the moduli space, which plays a key role in the work \cite{mirzakhani-annals} of Mirzakhani. We first recall the definitions of the Teichm\"uller metric, the Teichm\"uller flow and the Teichm\"uller volume~; we note that they extend to the nonorientable setting. 

\section{Preliminaries on Teichm\"uller geometry}\label{sec:teichmuller}

\subsection*{The Teichm\"uller metric in the orientable setting}
Let $T$ be a smooth closed oriented surface of genus $g\geq 2$. 

\subsubsection*{Quadratic differentials}
Given a complex structure $X$ on $T$, a \emph{quadratic differential} is a tensor $q$ which is locally of the form $q_z=f(z) \diff z^2$ with $f$ holomorphic. The Riemannian metric $|q|$ is flat with holonomy in $\{\pm \mrm{Id}\}$ and conical singularities at the zeros of $q$.
Equivalently $(T,|q|)$ is isometric to the quotient of a polygon of $(\C,|\diff z|^2)$ by a pairing of the sides realized by isometries of the form $z\mapsto \pm z+c$ ($c\in\C$). There are two preferred orthogonal geodesic foliations on $(T,|q|)$: the horizontal and the vertical ones which are respectively given by the real $1$--forms $\real(q)$ and $\im(q)$. Each $1$--form defines a transverse measure on the corresponding foliation. We abusively denote by $\real(q)$ and $\im(q)$ these transverse measured foliations. The couple $(X,q)$ is completely determined by the pair $(\real(q), \im(q))$ or by the polygon $P\subset (\C,\diff z^2)$ and the pairing of its sides.

\subsubsection*{Bundle of quadratic differentials}
The bundle $\qteich(T)\rightarrow \teich(T)$ of isotopy classes of pairs $(X,q)$ is called the \emph{bundle of quadratic differentials}.\par

 We denote by $\qteich(T;a_1,\ldots,a_k)$ the subset of $\qteich(T)$ that consists in quadratic differentials whose zeros have multiplicities $a_1,\ldots,a_k$. From the Gauss--Bonnet formula we have $-2\chi(T)=a_1+\ldots+a_k$. The subset $\qteich(T;a_1,\ldots,a_k)$ is a smooth submanifold of dimension $-2\chi(T)+2k$. In particular $\qteich(T;1,\ldots,1)$ is a dense open subset of full measure (with respect to the Lebesgue class).\par
 
  Let us denote by $Z$ the set of zeros of a quadratic differential $q\in\qteich(T;1,\ldots,1)$. The flat surface $(T,|q|)$ admits a geodesic triangulation whose set of vertices is $Z$. Let us choose a collection of edges of the triangulation $(e_1,\ldots,e_{-2\chi(T)})$ which realizes a basis of $\Hgy_1(T,Z;\R)$. We call $q$--\emph{holonomy} of $e_i$ the quantity 
 $$\mrm{hol}_q(e_i)=\int_{e_i} \sqrt{q},$$ 
which depends on the choice of a branch of $\sqrt{q}$ and of an orientation of $e_i$. Any quadratic differential $q'$ sufficiently close to $q$ admits a geodesic triangulation in the same isotopy class (relative to $Z$). So that one can define a continuous map $q'\mapsto (\hol_{q'}(e_i))_i$ on a neighborhood of $q$ in $\qteich(T)$. This map is actually a local diffeomorphism. Moreover the set of such maps forms an atlas for a piecewise linear integral structure on $\qteich(T;1,\ldots,1)$. Note that $q$ corresponds to an integral point if and only if 
$\hol_q(e_i)\in\Z\oplus i\Z$ for any $i$. The induced notion of volume, called \emph{Teichm\"uller volume}, is locally given by 
$$\bigwedge_i \diff \real(\hol_q(e_i))\wedge \diff\im(\hol_q(e_i)).$$

\subsubsection*{From quadratic differentials to measured laminations}
The following map is a $\Mod(S)$--invariant homeomorphism
$$\begin{array}{clc}
\qteich(T) & \longrightarrow &  \ML(T)\times \ML(T)-\Delta \\
 q & \longmapsto & (\real(q), \im(q))
\end{array}$$ 
where
$$\Delta=\{(\l,\eta)~;~i(\g,\l)+i(\g,\eta)=0\textnormal{ for some } \g\in\ML(T)\}.$$
Mirzakhani (\cite[Lemma~4.3]{mirzakhani-imrn}) showed that the restriction of this map to $\qteich(T;1,\ldots,1)$ is a piecewise integral linear isomorphism, in particular it preserves the volume. Note that $|\im \hol_q(e_i)|=i(e_i,\real(q))$ and $|\real \hol_q(e_i)|=i(e_i,\im(q))$.\par

 In the sequel we adopt the point of view of measured laminations jusitified by the above isomorphism. One advantage of this point of view is that all notions extends immediately to all surfaces of finite type with negative Euler characteristic.
 
\subsubsection*{The Teichm\"uller metric}
 There is a canonical identification between $\qteich(T)$ and the cotangent bundle of $\teich(T)$, whereas the bundle of Beltrami differentials on $T$ identifies with the tangent bundle of $\teich(T)$. The pairing between Beltrami and quadratic differentials induces an isomorphism between the corresponding bundles. Therefore one can define a Finsler metric on $\teich(T)$ through $\qteich(T)$.\par
 
  The \emph{Teichm\"uller metric} is the Finsler metric on $\teich(T)$ defined by the norm 
 $$\|q \|=\int_T |q|=\mrm{area}(q)=i(\real(q),\im(q)).$$
 We set 
 $$\uqteich(T)=\{q\in\qteich(T)~;~\|q\|=1\}.$$
The  \emph{Teichm\"uller flow} is the geodesic flow of the Teichm\"uller metric.
In terms of polygons, the Teichm\"uller geodesic passing through $q\in\qteich(T)$ at $t=0$ is given by
$$t\mapsto\begin{pmatrix} e^t & 0\\ 0 & e^{-t} \end{pmatrix}\cdot P\quad \textnormal{for all }t\in\R,$$
where $P$ is a polygon that represents $q$. In terms of measured foliations, the same Teichm\"uller geodesic is given by
$t\longmapsto (e^{t}\real(q),e^{-t} \im(q))$ for all $t\in\R$.\par
 The trajectory of the \emph{Teichm\"uller horocyclic flow} passing through $q\in\qteich(T)$ at $t=0$ is given by
$$t\mapsto\begin{pmatrix} 1 & t\\ 0 & 1 \end{pmatrix}\cdot P\quad \textnormal{for all }t\in\R,$$
In terms of measured laminations, the Teichm\"uller horocyclic flow corresponds to the earthquake flow (see \cite{mirzakhani-imrn}).

\subsection*{Definitions in the nonorientable setting}
 Let $S$ be a closed nonorientable surface with $\chi(S)<0$. Then the bundle of quadratic differentials $\qteich(S)$ is the set of isotopy classes of pairs $(X,q)$ where $X$ is a dianalytic structure on $S$ and $q$ a quadratic differential with respect to $X$. We recall that a \emph{dianalytic structure} is given by an atlas whose changes of charts are holomorphic or anti--holomorphic.\par

\subsubsection*{The flat surface $(X,|q|)$} It is isometric to the quotient of a polygon $P\subset(\C,|dz|^2)$ by a pairing of the sides realized by isometries of the form $z\mapsto \pm z+c$ or $z\mapsto \pm\bar z+c$ with $c\in \C$. Thus the holonomy is not in $\{\pm \mrm{Id}\}$ anymore, and there are geodesics (possibly closed) with self--intersections outside the singularities. However the holonomy preserves the horizontal and vertical directions, thus the horizontal and vertical geodesics do not self--intersect outside the singularities, and there are two measured foliations $\real(q)$ and $\im(q)$.\par
 
\subsubsection*{The Teichm\"uller flow} 
  The linear part of the isometry $z\mapsto \pm \bar z+c$ is diagonal, thus it commutes with the linear map $(x,y)\mapsto (e^t,e^{-t})$.
It follows that the identifications of the sides of the polygon
$$\begin{pmatrix} e^t & 0\\ 0 & e^{-t} \end{pmatrix}\cdot P\subset (\C,\diff z^2)$$
are of the form $z\mapsto \pm z+c$ or $z\mapsto \pm\bar z+c$. Therefore the polygon and the pairing of the sides determine a quadratic differential on $S$. This shows that the Teichm\"uller flow is well--defined on $\qteich(S)$. It is still given by $t\mapsto (e^t\real(q),e^{-t}\im(q))$ in terms of measured laminations.\par

\subsubsection*{The horocyclic flow} 
On the contrary there is no Teichm\"uller horocyclic flow on $\qteich(S)$. Indeed, the conjugate of $z\mapsto \pm \bar z +c$ by $(x,y)\mapsto (x+ty,y)$ is not an isometry whenever $t\neq 0$.\par

 As well--known, the Teichm\"uller space of the Klein bottle can be identified with the geodesic $i\R^\ast_+$ of $\Hyp$, where $\Hyp$ has to be understood as the Teichm\"uller space of the torus equipped with its Teichm\"uller metric. Clearly the horocyclic flow of $\Hyp$ does not stabilize $i\R_+^\ast$.

\subsubsection*{Volume and piecewise integral linear structure}
These structures extends readily to the nonorientable setting. As in the orientable case, it is possible to define them through the holonomy of quadratic differentials, or through the real and imaginary measured foliations.\par 

\subsubsection*{The orientation cover point of view}
Let us denote by $\widehat S$ the orientation cover of $S$, and by $\f$ its automorphism. Any quadratic differential $(X,q)$ on $S$ lifts to a quadratic differential $(\widehat X,\hat q)$ on $\widehat S$. The automorphism $\f$ changes $\hat q$ into its conjugate. The map $q\mapsto \hat q$ identifies $\qteich(S)$ with the fixed--point locus $\fix(\f)\subset \qteich(\widehat S)$ of the mapping class $[\f]$. The bundle structure of $\qteich(S)$ is the one induced by $\qteich(\widehat S)$ on $\fix(\f)$.\par

\subsection*{Volume on $\uqmoduli(S)$}
 The \emph{moduli space $\qmoduli(S)$ of quadratic differentials on $S$} is the quotient $\qteich(S)/\Mod(S)$. Similarly the \emph{moduli space of unit area quadratic differentials on $S$} is $\uqmoduli(S)=\uqteich(S)/\Mod(S)$.\par
 
  We define a Borel measure $\vol_1$ on $\uqteich(S)$ as follows: we set 
$$\vol_1(U)=\vol((0,1)\cdot U)\quad \textnormal{for any measurable } U\subset \qteich^1(S)$$ 
where $\vol$ is the Teichm\"uller volume on $\qteich(S)$. Since the action of $\Mod(S)$ on $\uqteich(S)$ is proper and discontinuous, the measure $\vol_1$ induces a measure on $\uqmoduli(S)$ still denoted by $\vol_1$.\par

\section{Teichm\"uller volume}
In this section we prove the following theorem:

\begin{theorem}
Let $S$ be a nonorientable surface of finite type with $\chi(S)<0$.
The moduli space $\uqmoduli(S)$ of unit area quadratic differentials on $S$ has infinite volume.
\end{theorem}

Actually we are going to prove a more precise result: 

\begin{proposition}\label{pro:volume-balls}
If $\g$ and $\d$ are two maximal one--sided simple multicurves that fill $S$.
Then the projection of $\cone(\g)\times\cone(\d)$ in $\uqmoduli(S)$ has infinite volume.
\end{proposition}

A multicurve is \emph{one--sided} if each of its components is one--sided. A one--sided simple multicurve $\g=\g_1+\ldots+\g_n$ is \emph{maximal} if its complement $S-\g$ is orientable. Two simple multicurves $\g$ and $\d$ \emph{fill} $S$ if $i(\g,\eta)+i(\d,\eta)>0$ for any simple closed geodesic $\eta$. Note that $\g$ and $\d$ can not share a component.

\begin{proof}[Proof of Proposition~\ref{pro:volume-balls}]
The projection $\pi:\qteich(S)\rightarrow\qmoduli(S)$ is a covering with negligible ramification locus. Its restriction $\cone(\g)\times\cone(\d)\rightarrow\pi(\cone(\g)\times\cone(\d))$ is a ramified covering of finite degree $d\geq 1$ (Lemma~\ref{lem:finite-fibers}). Thus the volume of the projection of $\cone(\g)\times\cone(\d)$ on $\uqmoduli(S)$ is
$\frac{1}{d}\cdot \vol(\{(\l,\mu)\in \cone(\g)\times \cone(\d)~;~i(\g,\d)\leq 1\})$. But this quantity is infinite (Lemma~\ref{lem:infinite-volume}).\end{proof}

Let us denote by $\pi$ the projection $\pi:\qteich(S)\rightarrow\qmoduli(S)$. This is a ramified covering since the action of $\Mod(S)$ on $\qteich(S)$ is a proper and discontinuous. Here \emph{ramified} means that, in the neighborhood of a singular point, the projection $\pi$ is an \emph{\'etale covering} up to the action of a finite group. 

\begin{lemma}\label{lem:finite-fibers}
The restriction of $\pi$ to $\cone(\g)\times\cone(\d)\rightarrow \pi(\cone(\g)\times\cone(\d))$ is a ramified covering of finite degree.
\end{lemma}

\begin{proof}
Two elements in $\cone(\g)\times\cone(\d)$ have same image in $\qmoduli(S)$ if and only if they belong to the same $\Mod(S)$--orbit. Thus it suffices to show that the number of $f\in \Mod(S)$ such that $f(\cone(\g))\cap \cone(\g)\neq \emptyset$ and $f(\cone(\d))\cap \cone(\d)\neq \emptyset$ is finite.\par
 Any $f\in\Mod(S)$ satisfies the alternative $f(\g)=\g$ or $i(f(\g),\g)\neq 0$, and the similar alternative with $\d$. This is an easy consequence of the fact that $\g$ and $\d$ are maximal among families of disjoint one--sided simple closed geodesics. We deduce that if $f\in\Mod(S)$ identifies two elements in $\cone(\g)\times\cone(\d)$, then $f$ fixes $\g$ and $\d$.\par
 
 The complement $S-(\g\cup\d)$ consists in a finite number of finite sided polygons that may have a puncture or a hole. If there is a hole then the boundary of the polygon is a component of $\partial S$. The number of such polygons is finite. We deduce that the number of $f\in \Mod(S)$ fixing $\g$ and $\d$ is finite. In view of the above alternative, this is equivalent to say that the number of $f\in\Mod(S)$ such that $f(\cone(\g))\cap \cone(\g)\neq \emptyset$ and $f(\cone(\d))\cap \cone(\d)\neq \emptyset$ is finite.
\end{proof}

\begin{lemma}\label{lem:infinite-volume}
The volume of $\{(\l,\mu)\in \cone(\g)\times \cone(\d)~;~i(\g,\d)\leq 1\}$ is infinite.
\end{lemma}

\begin{proof}
We set
\begin{eqnarray*}
V & = & \{v \in \cone(\d)~;~ i(\g,v)< 1/2\},\\
U & = & \left\{u \in\ML(S-\g)~;~  \sup_{v\in V} i(u,v)<1/2\right\},
\end{eqnarray*}
and 
\begin{eqnarray*}
W_t & = & (t\cdot (\g+U))\times (t^{-1}\cdot V),\\
W & = & \bigcup_{t>0} W_t.
\end{eqnarray*}
We note that $W\subset\cone(\g)\times\cone(\d)$, and that the union is disjoint (consider the weight of $\g$ on the first component). 
By construction, any $(\l,\mu)\in W$ satisfies $i(\l,\mu)<1$. Thus it suffices to show that $W$ has infinite volume to prove the lemma. Using the splitting of the Thurston measure (\textsection\ref{sec:decomposition}) we find
$$\vol (W) =\int_0^{+\infty} \vol(W_t)\ \mrm dt=\vol(U) \vol(V)~\int_{0}^{+\infty} t^{-n}  \mrm d t =+\infty,$$
where $n>0$ is the number of components of $\g$. We still have to show that $W$ is measurable, to do this it suffices to prove that $U$ and $V$ are open. The set $V$ is open by continuity of $i(\g,\cdot)$. It is also relatively compact in $\ML(S)\cup\{0\}$ since $\g$ and $\d$ fill up $S$ (use \cite[Proposition~8.2.25]{martelli}). Now we prove that $U$ is a nonempty relatively compact open subset of $\ML(S-\g)\cup\{0\}$.\par

 Let us pick $u\in \ML(S-\g)$, the supremum $t_u=\sup_{v\in V} i(u,v)$ is finite by relative compactness of $V$, so that $(3t_u)^{-1}u \in U$. This shows that $U$ is nonempty. Clearly $\frac{1}{3i(\g,\d)} \d\in V$, thus $\frac{i(u,\d)}{3i(\g,\d)}\leq \sup_{v\in V} i(u,v)$ and 
$$i(u,\g)+i(u,\d)=i(u,\d)< \frac{3i(\g,\d)}{2} $$
 for any $u\in U$. This implies that $U$ is relatively compact in $\ML(S-\g)\cup\{0\}$ since $\g$ and $\d$ fill up $S$. It remains to show that $U$ is open.\par
 
  Let $K$ be a compact subset of $\ML(S)\cup\{0\}$ whose interior contains $\overline{U}$. We consider the family $\{u\mapsto i(u,v)\}_{v\in V}$ of continuous functions from $K$ to $\R$. These functions are $L$--Lipschitz continuous for some constant $L$ that depends on $V$. Indeed the intersection function $i:\ML(S)\times \ML(S)\rightarrow\R_+$ is Lipschitz (Rees \cite[Corollary~1.11]{rees}, Luo and Stong \cite[Theorem~1.1]{luo}) with respect to some natural metrics (for instance they use the norm of the Dehn--Thurston coordinates in \cite{luo}). According to Ascoli's theorem, the family $\{u\mapsto i(u,v)\}_{v\in V}$ is relatively compact in $C(K,\R)$ equipped with the uniform norm, therefore $u\mapsto \sup_{v\in V} i(u,v)$ is continuous over $K$. We conclude immediately that $U$ is open.
\end{proof}

\begin{proposition}
Let $S$ be a closed nonorientable surface with $\chi(S)<0$. Then the Teichm\"uller geodesic flow on $\uqmoduli(S)$ is not topologically transitive, in particular is not ergodic with respect to any Borel measure of full support.
\end{proposition}

\begin{proof}
We identifiy $\uqmoduli(S)$ with $\{(\l,\mu)\in\ML(S)\times\ML(S)~;~i(\l,\mu)=1\}$. We write $\l$ and $\mu$ in the form $\l=\l^-+\l^+$ and $\mu=\mu^-+\mu^+$ where $\l^+,\mu^+\in\ML^+(S)$ and $\l^-,\mu^-$ are one--sided simple closed multicurves. It suffices to remark that the four continuous functions $(\l,\mu)\mapsto i(\l^\pm,\mu^\pm)$ over $\ML(S)\times \ML(S)$ are invariant under the Teichm\"uller flow but not globally constant.
\end{proof}

\section{Weil--Petersson volume}\label{sec:weil-petersson}

\subsection*{The orientable case}
 Let $T$ be a closed oriented surface with $\chi(T)<0$. Its Teichm\"uller space $\teich(T)$ admits a $\Mod(T)$--invariant K\"ahlerian metric called the \emph{Weil--Petersson metric}. Wolpert showed that the Fenchel--Nielsen coordinates are Darboux coordinates for the Weil--Petersson symplectic form $\omega_{WP}$. More precisely, given any pants decomposition $\g=\{\g_1,\ldots,\g_n\}$ of $S$, we have (see \cite[\textsection8.3]{imayoshi}):
\begin{eqnarray}\label{eq:wolpert}
\omega_{WP} & = & \sum_{i=1}^n\ \mrm d \tau_i\wedge \mrm d \ell_i,
\end{eqnarray}
where $(\tau_1,\ell_1,\ldots,\tau_n,\ell_n)$ are the Fenchel--Nielsen coordinates associated to $\g$. More generally the formula~\eqref{eq:wolpert} defines a $\Mod(T)$--invariant symplectic form on $\teich(T)$ for any compact oriented surface $T$ with $\chi(T)<0$.\par
 Let $\nu_{WP}$ be the $\Mod(T)$--invariant volume form which is the $n$--th exterior power of $\omega_{WP}$. In the Fenchel--Nielsen coordinates we have:
\begin{eqnarray}\label{eq:volume-WP}
\nu_{WP} &  = &  \bigwedge_{i=1,\ldots,n} \mrm d \tau_i \wedge\mrm d \ell_i.
\end{eqnarray}
The Weil--Petersson volume of the moduli space $\moduli(T)$ is finite.  Mirzakhani (\cite{mirzakhani-inventiones}) found a recursive formula for $\nu_{WP}(\moduli(T))$ when $\partial T\neq \emptyset$, and deduced that $\nu_{WP}(\moduli(T))$ is a polynomial of degree $\dim\moduli(T)$ in the lengths of the boundary components. She also gave a formula for the average $\int_{\moduli(T)} F(x) \diff\nu_{WP}$ of a function of the form $F=\sum_{f\in \Mod(T)} \ell_{f\cdot \g_0}$ where $\g_0$ is a simple closed geodesic (\cite[\textsection 7]{mirzakhani-inventiones}) . This is a key ingredient in the determination of the asymptotic of the number of simple closed geodesic of length less than $L$ (Step~3 in the proof of \cite[Theorem~6.4]{mirzakhani-annals}).\par 
 
\subsection*{The Teichm\"uller space as a Lagragian submanifold}
Let $S$ be a compact nonorientable surface with $\chi(S)<0$. We denote by $\widehat S$ its orientation cover whose automorphism is $\f:\widehat S\rightarrow\widehat S$. Note that $\f$ is an orientation reversing involution.
The map $m\mapsto \hat m$, that consists in lifting a hyperbolic metric on $S$ to $\widehat S$, induces a diffeomorphism between the Teichm\"uller space $\teich(S)$ and $\fix(\f)\subset\teich(\widehat S)$, the fixed--point locus of the mapping class $[\f]$.\par

 Let $\g$ be a pants decomposition which is stabilized by $[\f]$ (such that there is a permutation $\s$ of $\{1,\ldots , n\}$ satisfying $[\f]\cdot \g_i=\g_{\s(i)}$ for any $i=1,\ldots, n$). As $\f$ reverses orientation we have $[\f]^\ast \diff\tau_i=-\diff\tau_{\s(i)}$, therefore $[\f]^\ast \omega_{WP}=-\omega_{WP}$. So the fixed--point locus $\fix(\f)$ is a Lagrangian submanifold of $(\teich(\widehat S),\omega_{WP})$.\par

 Let us denote by $\pi$ the fundamental group of $\widehat S$. As well--known, there is a $\Mod(S)$--equivariant diffeomorphism between $\teich(\widehat S)$ and $\Rep(\pi,\PSL(2,\R))$, the space of faithfull and discrete representations $\rho:\pi\rightarrow\PSL(2,\R)$ up to conjugacy. Moreover, the tangent space $T_{[\rho]}\Rep(\pi,\PSL(2,\R))$ identifies with $\Hgy^1(\pi,\mathsf{sl}(2,\R))$, where $\mathsf{sl}(2,\R)$ is the Lie algebra of $\SL(2,\R)$ seen as a $\pi$--module with respect to the action $\Ad\circ \rho$. Goldman (\cite{goldman}) showed that the image of $\omega$ through the diffeomorphism $\teich(\widehat S)\rightarrow \Rep(\pi,\PSL(2,\R))$ is a multiple of the symplectic form
$$\Hgy^1(\pi,\mathsf{sl}(2,\R))\times\Hgy^1(\pi,\mathsf{sl}(2,\R))\rightarrow \Hgy^2(\pi,\R)\simeq \R,$$
induced by the cup product together with the Killing form of $\mathsf{sl}(2,\R)$. The set of representations invariant under the action of $[\f]$ forms a submanifold of $\Rep(\pi,\SL(2,\R))$ whose tangent space at $[\rho]$ is the set of fixed points of $\f_\ast$ in $\Hgy^1(\pi,\mathsf{sl}(2,\R))$. By naturality of the cup product, and because $\f_\ast$ sends the fundamental class to its opposite, it comes that $\f_\ast$ is an anti--symplectomorphism, so its set of fixed points is a Lagrangian submanifold. 
 
\subsection*{Norbury's volume form}
From the above paragraph, it seems rather difficult to find a $\Mod(S)$--invariant symplectic form on $\teich(S)$ (or maybe another structure) that generalizes $\omega_{WP}$. Note also that the dimension of $\teich(S)$ can be odd. However it might possible to find a volume form that generalizes $\nu_{WP}$. Indeed, Norbury (\cite{norbury}) suggested the following:
\begin{eqnarray*}
\nu_{N} &  = &  \left(\bigwedge_{\g_i \textnormal{ one--sided}} \coth(\ell_i) \diff \ell_i\right)\wedge \left( \bigwedge_{\g_i \textnormal{ two--sided}} \diff\tau_i \wedge\mrm d \ell_i \right).
\end{eqnarray*}
where $\{\g_1,\ldots,\g_n\}$ is a pants decomposition of $S$. He showed that the \emph{Norbury's volume form} $\nu_N$ is $\Mod(S)$--invariant up to sign, in particular its absolute value is a $\Mod(S)$--invariant measure.\par

One can wonder in what respect $\nu_n$ is a generalization of $\nu_{WP}$. What guided Norbury is the formula \eqref{eq:volume-WP}, and he was looking for a measure given by a similar formula. In the next paragraph we provide a more conceptual justification.

\subsection*{A characterization by mean of the twist flow}
Let $T$ be compact oriented surface with $\chi(T)<0$. The \emph{twist flow} is simply defined by twisting a hyperbolic metric along a simple closed geodesic. It comes directly from the formula \eqref{eq:wolpert} of Wolpert that 
$ \omega_{WP}(\frac{\partial}{\partial\tau_\g},\cdot) =  \diff\ell_\g$
for any simple closed geodesic $\g$. Here $\tau_\g$ is the vector field associated to the twist flow in the direction $\g$.
By definition, this means that the twist flow is the Hamiltonian flow of $\ell_\g$. In particular the twist flow is volume preserving. All these considerations extend to the earthquake flow, and also to the shearing flow (\cite{bonahon-sozen,bonahon-toulouse}). It is worth mentioning that the twist flow can be defined on spaces of representations (see \cite{goldman-inventiones} and \cite{palesi}).\par

 Let $S$ be a nonorientable surface of finite type with $\chi(S)<0$. In that case, the twist flow is only defined for \emph{two--sided} simple closed curves, and the earthquake flow is only defined for measured laminations that are \emph{contained in some orientable subsurface}. Moreover these flows are well--defined up to sign.\par

 From the discussion above, it appears that the twist flow --- and more generally the earthquake flow --- is a canonical volume preserving flow. So we logically look for a characterization of Norbury and Weil--Petersson volume forms in terms of the twist flow.

\begin{proposition}
Let $S$ be a surface of finite type with $\chi(S)<0$. If $S$ is orientable, then $\nu_{WP}$  is the unique (up to a multiplicative constant) volume form on $\teich(S)$ invariant under the twist flow. If $S$ is nonorientable with $\chi(S)<-1$, then $\nu_{N}$ is the unique (up to a multiplicative constant) volume form on $\teich(S)$ invariant under the twist flow.
\end{proposition}

\begin{remark}
Let $S$ be a nonorientable surface with $\chi(S)=-1$. Then there exists a simple closed geodesic $\g$ such that $i(\g,\d)=0$ for any two--sided simple closed geodesic $\d$. In particular $\ell_\d$ is constant along the trajectories of the twist flow. This shows that the statement is false for these surfaces.
 \end{remark}

\begin{proof}
From the defining formulas, we see that $\nu_{WP}$ and $\nu_N$ are invariant under the twist flow. So the problem is to show their uniqueness. We consider a volume form $\nu$ on $\teich(S)$ invariant under the twist flow. We write $\nu=f\nu_{WP}$ or $\nu=f\nu_N$, where $f:\teich(S)\rightarrow\R$ is smooth and invariant under the twist flow.\par
We first treat the case $S$ orientable. We note that $f$ is invariant under the earthquake flow. This comes directly from the following facts: $f$ is continuous, the earthquake flow $\ML(S)\times \teich(S)\rightarrow \teich(S)$ is continuous, and multiples of simple closed geodesics are dense in $\ML(S)$. 
Then, we apply a classical theorem of Thurston (see \cite{kerckhoff}) which states that any two points in $\teich(S)$ are related by an earthquake path, and we conclude that $f$ is constant on $\teich(S)$.\par

 Now we assume $S$ nonorientable. We prove below that any two points in $\teich(S)$ are related by a combination of at most two earthquake paths (with respect to measured laminations contained in orientable subsurfaces), this finishes the proof.\par

Let $\g$ be a simple closed geodesic of $S$ whose complement is orientable. Let $\d$ be a two--sided simple closed geodesic that intersects $\g$ (such a geodesic exists because $\chi(S)<-2$). By twisting along $\d$ we can increase arbitrarily the length of $\g$, and thus join any two fibers of $\ell_\g$. Any fiber of $\ell_\g$ is canonically identified to the Teichm\"uller space of $S-\g$, thus we can apply the theorem of Thurston mentioned above to join any two points on the same fiber.
\end{proof}
 
 It would be interesting to have other characterizations of $\nu_N$, for instance in terms of spaces of representations (see \cite{palesi} for representations of nonorientable surface groups into compact Lie groups) or shearing coordinates (see \cite{mirzakhani-imrn} for the relations between various notions of volumes on moduli spaces).
 
\subsection*{Infinite volume}
The aim of this paragraph is to explain the following theorem due to Norbury (\cite{norbury}). The ideas contained in this paragraph are also known to Yi Huang.

\begin{theoremnonumber}[Norbury]
If $S$ is a nonorientable surface of finite type with $\chi(S)<0$, then $\nu_N(\moduli(S))$ is infinite.
\end{theoremnonumber}

Norbury's proof mimic Mirzakhani's computation of Weil--Petersson volumes, in particular it relies on an identity \emph{\`a la} McShane. So it does not really explain why the volume is infinite. Here we determine rather precisely which part of $\moduli(S)$ has infinite volume. The first step of the proof is similar to Lemma~\ref{lem:finite-fibers}, that is we work with a finite covering.\par

 Our proof goes as follows: we consider the subset of $\moduli(S)$ that consists in hyperbolic surfaces having a short one--sided multicurve $\g=\g_1+\ldots+\g_n$ whose complement $S-\g$ is orientable, we show that this subset fibers in Teichm\"uller spaces $\teich(S-\g)$ of orientable subsurfaces. We know the volume of the fibers thanks to Mirzakhani's results, and the transverse measure is obviously given by $\coth(\ell_1)\cdots\coth(\ell_n)\diff\ell_1\cdots\diff\ell_n$, so that we can explicitely compute the volume.

\begin{proof}[Proof of Norbury's theorem]
 Let $\g=\g_1+\ldots+\g_n$ be a maximal family of disjoint one--sided simple closed geodesics. Note that $S-\g$ is orientable. For any $\e>0$ we denote by $\mc U_{\g}(\e)$ the projection in $\moduli(S)$ of the open subset
\begin{eqnarray*}
U_\g(\e) & = & \{[m]\in\teich(S)~;~\ell_{\g_i}(m)<\e\textnormal{ for } i=1,\ldots,n \}.
\end{eqnarray*}
By Lemma~\ref{lem:WP-infinite} the volume $\nu_N(\mc U_{\g}(\e))$ is infinite.
\end{proof}

\begin{lemma}
For $\e>0$ small enough, $U_\g(\e)/\Mod(S-\g)$ is a finite ramified covering of $\mc U_\g(\e)$.
\end{lemma}

\begin{proof}
By the Collar Lemma (see \cite[Chap.4]{buser}), for any point in $U_\g(\e)$ the set of one--sided simple closed geodesics of length less than $\e$ is precisely $\g$. This implies that an element of $\Mod(S)$ that preserves $U_\g(\e)$ stabilizes $\g$. Thus the subgroup of $\Mod(S)$ that preserves $U_\g(\e)$ is $\Mod^\ast(S-\g)$ and $\mc U_\g(\e)=U_\e(\g)/\Mod^\ast(S-\g)$.
But $\Mod(S-\g)$ is a subgroup of finite index of $\Mod^\ast(S-\g)$. So $U_\g(\e)/\Mod(S-\g)$ is a finite covering of $\mc U_\g(\e)$.
\end{proof}

\begin{lemma}\label{lem:WP-infinite}
The $\nu_N$--volume of $\mc U_\g(\e)$ is infinite.
\end{lemma}

\begin{proof}
According to the above lemma, it suffices to show that $U_\g(\e)/\Mod(S-\g)$ has infinite volume. 
As $\Mod(S-\g)$ acts trivially on $\g$, the family of length functions $(\ell_{\g_1},\ldots,\ell_{\g_n})$ induces a map $L:U_\g(\e)/\Mod(S-\g)\rightarrow(0,\e)^n$. Each fiber $L^{-1}(x)$ identifies canonically with $\moduli(S-\g,x)$ whose volume $V_{S-\g}(x)$ is a polynomial of degree $\dim \teich(S-\g)$ in $x$ (Mirzakhani \cite{mirzakhani-inventiones}). We find
$$
\vol(U_\g(\e))/\Mod(S-\g))\ =\ \int_{(0,\e)^n} \coth(x_1)\cdots\coth(x_n) V_{S-\g}(x)\ \mrm dx \ =\ +\infty,
$$
which concludes the proof.
\end{proof}

It is interesting to note that if $\g=\g_1+\ldots+\g_n$ and $\d=\d_1+\ldots+\d_m$ are maximal families of disjoint one--sided simple closed geodesics, then for any $\e$ small enough either $\mc U_\g(\e)= \mc U_\d(\e)$ or $\mc U_\g(\e)\cap\mc U_\d(\e)=\emptyset$, the first case occurring if and only if $\g$ and $\d$ have same topological type (\emph{i.e.} are in the same $\Mod(S)$--orbit). This fact is an obvious consequence of the Collar Lemma. One can easily compute the number of such topological types which is the integer part of $(g+1)/2$ where $g$ is the genus of $S$.\par

\section{A finite volume deformation retract}\label{sec:finite-convex}

We call \emph{systole of one--sided geodesics} the length of the shortest one--sided simple closed geodesic. This metric invariant defines a continuous and $\Mod(S)$--invariant function $\sys^-:\teich(S)\rightarrow \R_+^\ast$. In this section we study the $\Mod(S)$--invariant subset
\begin{eqnarray*}
\teich^-_\e(S) & = & \{[m]\in\teich(S)~;~\sys^-(m)\geq\e\}
\end{eqnarray*}
and its quotient $\moduli_\e^-(S)=\teich_\e^-(S)/\Mod(S)$.

\subsection*{Noncompact subsets of finite volume}

\begin{proposition}
For any $\e>0$ the subset $\moduli_\e^-(S)$ has finite $\nu_N$--volume. It is noncompact if $\e$ is sufficiently small and if $S$ is not the two--holed projective plane.
\end{proposition}

\begin{proof}
Using standard methods (see \cite[\textsection 5]{buser}), one easily shows that there exists a constant $B(S)$, called the \emph{Bers' constant}, such that any point in $\teich(S)$ admits a pants decomposition whose components have length at most by $B(S)$. Let us recall that, in our definition of $\teich(S)$, we assume that the lengths of the boundary components of $S$ are fixed.\par
 
 Given a pants decomposition $\g=\{\g_1,\ldots,\g_n\}$, we denote by $A_\g(\e)\subset\teich(S)$ the subset defined by the following inequalities in the Fenchel--Nielsen coordinates associated to $\g$: 
 $$\left\{\begin{array}{l}
 0\leq \tau_i\leq \ell_i \textnormal{ and } \ell_{i}\leq B(S)\textnormal{ for all } i=1,\ldots, n \\
 \ell_{i}\geq\e \textnormal{ for all } i=1,\ldots,n\textnormal{ such that }\g_i \textnormal{ is one--sided}.
 \end{array}\right.$$
We observe immediately that $A_\g(\e)$ has finite $\nu_N$--volume.
 
 Let us denote by $\mc A_\g(\e)$ the projection of $A_\g(\e)$ in $\moduli(S)$. The existence of the Bers' constant implies that any point in $\moduli_\e^-(S)$ belongs to some $\mc A_\g(\e)$ for some pants decomposition $\g$. But $\mc A_\g(\e)$ depends only on the \emph{topological type} of $\g$, and there are finitely many topological types of pants decomposition. So we can cover $\moduli_\e^-(S)$ by a finite number of subsets of the form $\mc A_\g(\e)$. We conclude that $\moduli_\e^-(S)$ has finite $\nu_N$--volume since every $A_\g(\e)$ has finite $\nu_N$--volume.\par
 
 If $S$ is not the two--holed projective plane, then $S$ has a pants decomposition with a two--sided component. By pinching the two--sided component while keeping the lengths of the other components equal to $2\e$ we leave any compact in $\moduli(S)$ (Mumford's compactness criterion) but remain in $\moduli_\e^-(S)$ (Collar Lemma). This proves the second assertion.
\end{proof}

\subsection*{Quasi--convexity}
In this paragraph we just ask the following natural question (see \textsection\ref{sec:conclusion} for some motivations):

\begin{question}
Is $\teich_\e^-(S)$ quasi--convex with respect to the Teichm\"uller metric~?
\end{question}

We believe that \emph{Minsky's product theorem} (\cite{minsky}) gives some evidence to a positive answer. It provides an approximation of the Teichm\"uller distance when the systole is small, but not an approximation of the Finsler metric, so we were not able to conclude.

\subsection*{A retraction}

Let us fix $\e>0$ small enough in the sense that two closed geodesics of length at most $\e$ can not intersect in any hyperbolic surface (independently of the topology).

\begin{proposition}
For $\e>0$ small enough there is a $\Mod(S)$--invariant (strong) deformation retraction of $\teich(S)$ onto $\teich^-_\e(S)$.
\end{proposition}

 It is well--known that for $\e$ small enough $\{\sys\geq\e\}$ is a $\Mod(S)$--invariant deformation retract of $\teich(S)$ (see \cite[\textsection 3]{ji} for a proof and a discussion of its relation with the well--rounded retract). The statement above is a little bit different since we deal with $\sys^-$ and not $\sys$. Our deformation retraction follows the flow of a well--chosen vector field that increases the length of the geodesics realizing $\sys^-$. This is rather classical, the additional difficulty is to make sure that the surface does not degenerate before we reach $\teich^-_\e(S)$. To do so we need to control also the length of the two--sided geodesics. This is done using strip deformations. 
 
\begin{proof}
Let us introduce some notations, for short we drop the $(S)$ in all notations involving $\teich(S)$. We denote by $\teich^{\leq n}_\e$ the set of $X\in\teich$ which have at most $n\geq 0$ one--sided geodesics of length less than $\e$. We have $\teich^{\leq 0}_\e=\teich^-_\e$ and $\teich^{\leq n}_\e=\teich$ for any $n$ greater or equal to the genus $g$ of $S$. Given a one--sided simple multicurve $\g=\g_1+\ldots+\g_n$ we denote by $\teich^\g_\e$ the set of $X\in\teich$ whose one--sided geodesics of length less than $\e$ are exactly $\g_1,\ldots,\g_n$. For short we denote by $\ell_i$ the length function $\ell_{\g_i}$. The subsets $\teich^\g_\e$ are pairwise disjoint.\par

  The strategy of the proof is to decrease the maximal number of one--sided simple closed geodesics of length less than $\e$. We successively retract
$$\teich=\teich^{\leq g}_\e\longrightarrow \teich_\e^{\leq g-1}\longrightarrow \ldots \longrightarrow \teich^{\leq 1}_\e\longrightarrow\teich^{\leq 0}_\e=\teich^-_\e$$
 The construction of the retraction $\teich^{\leq n}_\e\longrightarrow \teich_\e^{\leq n-1}$ is based on the following statement (Lemma~\ref{lem:retraction-2}) which is the heart of the proof: \emph{given a one--sided multicurve $\g=\g_1+\ldots+\g_n$ there exists a $\Mod^\ast(S-\g)$--invariant retraction }
 $$R^\g:\teich^{\leq n}_\e\longrightarrow  \teich^{\leq n}_\e-\teich^\g_\e.$$
 Once we have $R^\g$ we get for free a $\Mod(S)$--invariant retraction
 $$R^{[\g]}:\teich^{\leq n}_\e\rightarrow \teich^{\leq n}_\e-\left(\cup_{f\in \Mod(S)} \teich^{f(\g)}_\e\right),$$
 defined by 
 $$\left\{\begin{array}{ll}
 R^{[\g]}(X)=f(R^{\g}(X)) & \textnormal{if }X\in\teich_\e^{f(\g)}\textnormal{ for some }f\in\Mod(S),\\
R^{[\g]}(X)=X &  \textnormal{otherwise}.
\end{array}\right.$$
Note that $R^{[\g]}$ does not depend on the particular $f\in\Mod(S)$ such that $X\in\teich^{f(\g)}_\e$ since $R^\g$ is $\Mod^\ast(S-\g)$--invariant. The continuity of $R^{[\g]}$ follows directly from the continuity of $R^\g$ and the relative positions of the subsets $\teich^{f(\g)}_\e$, namely they are pairwise disjoint and relatively open in $\teich^{\leq n}$ (Lemma~\ref{lem:retraction-1}). Now we conclude easily: we perform the retraction $R^{[\g]}$ for each topological type $[\g]$ of one--sided multicurve with $n$ components (there are finitely many such topological types), this gives the expected retraction $\teich^{\leq n}_\e\longrightarrow \teich_\e^{\leq n-1}$. It can be checked that it is a deformation retraction like $R^\g$ (Lemma~\ref{lem:retraction-2}).
\end{proof}

\begin{lemma}\label{lem:retraction-1}
Let $\g=\g_1+\ldots+\g_n$ be a one--sided multicurve. The set $\teich^\g_\e$ is open in $\teich_\e^{\leq n}$.
\end{lemma}

\begin{proof}
We want to show that any $X_0\in\teich^\g_\e$ has a neighborhood $U$ in $\teich$ such that $U\cap\teich^\g_\e=U\cap\teich^{\leq n}_\e$. If $X_0$ belongs to the interior of $\teich^\g_\e$, then one obviously takes $U=\teich^\g_\e$. So we assume that $X_0\in \teich^\g_\e$ belongs to the frontier $\partial \teich^\g_\e$ of $\teich^\g_\e$ in $\teich$. Let $U$ be a neighborhood of $X_0$ in $\teich$ such that $\ell_i<\e$ on $U$ for any $i$. Then $U-\teich^\g_\e$ is the set of $X\in U$ such that $\ell_{\d}(X)<\e$ for some one--sided geodesic $\d$ which is not a component of $\g$. In particular $U-\teich^\g_\e\subset U-\teich^{\leq n}_\e$. We deduce that $U\cap\teich^\g_\e=U\cap\teich^{\leq n}_\e$, which is what we wanted to show.
\end{proof}

\begin{lemma}\label{lem:retraction-2}
Given a one--sided multicurve $\g=\g_1+\ldots+\g_n$ there exists a $\Mod^\ast(S-\g)$--invariant (strong) deformation retraction
$$R^\g:\teich^{\leq n}_\e\longrightarrow  \teich^{\leq n}_\e-\teich^\g_\e.$$
\end{lemma}

\begin{proof}
Let us complete $\g$ into a pants decomposition $\bar\g$. Let $v^{\bar\g}$ be the vector field on $\teich(S)$ given by
$v^{\bar\g}=\frac{\partial}{\partial \ell_1}+\ldots+\frac{\partial}{\partial \ell_n}$
in the Fenchel--Nielsen coordinates associated to $\bar\g$. Actually the twist coordinates are not canonically associated to a pants decomposition, but this does not matter here. It obviously satisfies
\begin{eqnarray}\label{eq:vector1}
\diff \ell_i(v^{\bar\g}) & = & 1 \quad (i=1,\ldots,n).
\end{eqnarray}
Its flow $\phi$ is explicitely given by
$
\phi(X;t)  =  X+(t,\ldots,t,0,\ldots,0)
$
with respect to the linear structure of the Fenchel--Nielsen coordinates. Note that $\phi$ is complete in positive time.\par

 The flow $\phi$ can be realized as a \emph{strip deformation}. This is a construction of Thurston which has been studied in details in \cite{danciger} (see also \cite{theret}). Given a surface with boundary, a strip deformation is parametrized by a system of arcs $\{\a_1,\ldots,\a_k\}$ together with a  point $p_j\in\a_j$ and a width $w_j>0$ for any $1\leq j\leq k$. For any hyperbolic metric $m$, the time $1$ strip deformation of $m$ in the direction $(\a,p,w)$ consists in cutting each $\a_j$, and inserting a strip of width $w_j$ in such a way that the geodesic segment contained in the strip and joining the points identified to $p_j$ is orthogonal to the boundary components of the strip. Actually this construction is realized in the Nielsen extension of the surface. We refer to \cite{danciger} for a more precise description. In our case, the surface with boundary is $S-\g$, the arcs $\a_j$'s are the common perpendicular between the $\g_i$'s with respect to the pants decomposition $\bar\g$, the points $p_j$'s are the midpoints of the common perpendiculars, the widths $w_j$'s are chosen so that \eqref{eq:vector1} is realized ($w_j$ depends on the metric). This point to view has the advantage that one can easily control the variation of the length of any closed geodesic $\d$ contained in $S-\g$. Looking at the trace of $\d$ in each pair of pants of $S-\bar \g$ we find (see also \cite[Lemma~2.2]{danciger})
 \begin{eqnarray}\label{eq:vector2}
0\ \leq \ \diff \ell_\d(v^{\bar\g}) & \leq & i(\d,\a_1)+\ldots+i(\d,\a_k).
\end{eqnarray}
We observe that $\diff \ell_\g(v^{\bar \g})$ is uniformly controlled over $\teich$. This gives another proof of the completeness of $\phi$ in positive time. We will use this argument later.\par

Unfortunately $\phi$ is not $\Mod(S-\g)^\ast$--invariant. To fix this problem we endow $\teich$ with the Teichm\"uller distance $d_\teich$, and choose a smooth function $\psi:\R\rightarrow [0,1]$ with $\psi(0)=1$ and $\psi\equiv 0$ outside $[0,1]$. Then we set
 \begin{eqnarray*}
 V_X & = & \sum_{f\in \Mod^\ast(S-\g)} \psi(d_\teich(X,f(X))) \cdot v^{f(\bar\g)}_X.
 \end{eqnarray*}
The sum is finite at any point $X\in\teich(S)$ since $d_\teich$ si complete (closed balls are then compact) and $\Mod(S)$ acts properly. So the vector field $V$ is well--defined and locally Lipschitz, in particular integrable. It is moreover $\Mod^\ast(S-\g)$--invariant:
\begin{eqnarray*}
 V_{g(X)}  & = & \sum_{f\in \Mod^\ast(S-\g)} \psi(d_\teich(g(X),f(X))) \ v^{f(\bar\g)}_{g(X)},\\
    & = &  \sum_{f\in \Mod^\ast(S-\g)} \psi(d_\teich(X,g^{-1}f(X))) \ \diff g_X(v^{g^{-1}f(\bar\g)}_{X}),\\
    & = &  \sum_{h\in \Mod^\ast(S-\g)} \psi(d_\teich(X,h(X))) \ \diff g_X(v^{h(\bar\g)}_{X}),\\
    & = & \diff g_X(V_X).
\end{eqnarray*}
We have used the equality $ \diff g_X(v^{g^{-1}f(\bar\g)}_{X})=v^{f(\bar\g)}_{g(X)}$ which comes from the fact that the Fenchel--Nielsen coordinates of $g(X)$ with respect to $f(\bar\g)$ are equal (up to some additive constants) to the Fenchel--Nielsen of $X$ with respect to $g^{-1}f(\bar\g)$. The additive constants are due to the fact that the twist coordinates are not canonically associated to a pants decomposition. \par

 Let $\Phi$ be the flow of the vector field $V$, we claim that $\Phi$ is complete in positive time. 
This is equivalent to say that, for any simple closed geodesic $\d$ and any $X\in\teich$, the length $\ell_\d(\Phi(X;t))$ can not tend to zero or infinity in a positive finite time.
From equations \eqref{eq:vector1} and \eqref{eq:vector2} and $1\geq \psi\geq 0$ we deduce that $\Phi(X,t)$ can not degenerate in a positive finite time.\par

 Using the flow $\Phi$ we construct the retraction $R^\g$ of the statement. For any $X\in\teich$ we set $t_X=\inf\{t\geq 0~;~\Phi(X;t)\notin\teich^\g_\e\}$. This number is finite since the lengths of the $\g_i$'s increase at least linearly (equation \eqref{eq:vector1} and $\psi(0)=1$). Moreover $X\mapsto t_X$ is continuous since
$t_X=\inf\{t\geq 0~;~\ell_i(\Phi(X;t))\geq \e\textnormal{ for any }i=1,\ldots n\}$ (the lengths of the geodesics disjoint from $\g$ increase). We conclude that the map $X\mapsto \Phi(X;t_X)$ gives the expected deformation retraction.
 \end{proof}

\subsection*{The frontier of $\teich^-_\e(S)$}

\begin{proposition}
For $\e>0$ small enough, the frontier of $\teich^-_\e(S)$ in $\teich(S)$ is homotopy equivalent to the geometric realization of $\curve^-(X)$.
\end{proposition}

\begin{remark}
\begin{enumerate}
\item The frontier of $\teich^-_\e(S)$ is simply $\{\sys^-=\e\}$.
\item As well--known $\{\sys=\e\}$ is homotopy equivalent to the complex of curves $\curve(S)$ (see \cite{ivanov}).
\end{enumerate}
\end{remark}

\begin{proof}
We use the standard idea, that is we construct a cover of $\{\sys^-=\e\}$ by contractible open subsets  (\emph{i.e} a \emph{good cover}) whose nerve is isomorphic to $\curve^-(S)$. Then we conclude (using \cite[Corollary~4G.3]{hatcher-book}) that $\{\sys^-=\e\}$ is homotopy equivalent to $\curve^-(S)$. Note that $\{\sys^-=\e\}$ is homeomorphic to a CW--complex.\par

Given a one--sided multicurve $\g=\g_1+\ldots+\g_n$ we denote by $U_\g$ the set of
$X\in\{\sys^-=\e\}$ such that $\ell_\d(X)>\e$ for any one--sided geodesic $\d$ disjoint from $\g$. Alternatively
$U_\g=\{\sys^-=\e\} \cap \{\min_{\d} \ell_\d>\e\} $
where $\d$ runs over the set of one--sided geodesics disjoint from $\g$. In this form we see that $U_\g$ is relatively open in $\{\sys-=\e\}$ since $\min_\d \ell_\d$ is continuous. Let us show that $U_\g$ is also contractible.\par

 The open subset $\{\min_\d \ell_\d>\e\}\subset\teich$ (where $\d$ should be taken as previously) is contractible. Actually one can prove that $\{\min_\d \ell_\d\geq2\e\}$ is a deformation retract of both $\teich$ and $\{\min_\d \ell_\d>\e\}$. We realize the deformation retractions through the flow of a well--chosen vector field, as in the proof of Lemma~\ref{lem:retraction-2}, but in a much simpler way since we do not require the $\Mod(S)$--invariance.\par
 
 To prove that $U_\g$ is contractible we show that it is a deformation retract of $\{\min_\d \ell_\d>\e\}$. We proceed as in beginning of the proof of Lemma~\ref{lem:retraction-2}. We complete $\g$ into a pants decomposition $\bar \g$. Then we consider the flow $\phi$ of the vector field $v^{\bar\g}=\frac{\partial}{\partial \ell_1}+\ldots+\frac{\partial}{\partial \ell_n}$. As we have seen, this flow increases the length of all geodesics disjoint from $\g$, in particular it preserves the open set 
$\{\min_\d \ell_\d>\e\}.$
We set 
$t_X=\min\{t\geq 0~;~\ell_i(\phi(X,t))\geq \e \textnormal{ for } i=1,\ldots, n\}.$
Clearly $X\mapsto t_X$ is a well--defined and continuous function, so $(X,s) \mapsto \phi(X, st_X)$ is a deformation retraction. This shows that $U_\g$ is contractible.\par

 The family $\{U_\g~;~\g\textnormal{ is a one--sided muticurve}\}$ is a good open cover of $\teich^-_\e$, its nerve is isomorphic to $\curve^-(S)$ due to the equality $U_\g\cap U_{\g'}=U_{\g\cap \g'}$ for any simple one--sided multicurves $\g$ and $\g'$.
So $\teich^-_\e$ is homotopy equivalent to $\curve^- (S)$. 
\end{proof}


\nocite{*}
\bibliographystyle{alpha}
\bibliography{biblio}


\end{document}